\documentclass[12pt]{iopart} 
\usepackage{iopams}

\usepackage{amstext}
\usepackage{amsthm}
\usepackage{amssymb}
\usepackage{dsfont}
\usepackage{hyperref} 
\usepackage{enumitem}
\expandafter\let\csname equation*\endcsname=\relax
\expandafter\let\csname endequation*\endcsname=\relax
\usepackage{amsmath}
 
\newtheorem{theorem}{Theorem}[section]
\newtheorem{proposition}[theorem]{Proposition}
\newtheorem{definition}[theorem]{Definition} 
\newtheorem{remark}[theorem]{Remark}
\newtheorem{lemma}[theorem]{Lemma} 
\newtheorem{corollary}[theorem]{Corollary}

\newcommand{\var}{\operatorname{var}}
\newcommand{\contrho}{\tilde{\rho}}
\newcommand{\conttau}{\tilde{\tau}}
\newcommand{\cov}{\operatorname{cov}}
\newcommand{\floor}[1]{\left\lfloor #1 \right\rfloor}
\newcommand{\ceil}[1]{\left\lceil #1 \right\rceil}

\newcommand{\aL}{a}
\newcommand{\bL}{b}
\newcommand{\sA}{\mathcal{A}}

\newcommand{\sG}{\mathcal{G}}

\newcommand{\sN}{\mathcal{N}}
\newcommand{\sP}{\mathcal{P}}
\newcommand{\sQ}{\mathcal{Q}}
\newcommand{\sR}{\mathcal{R}} 
\newcommand{\sS}{\mathcal{S}}
\newcommand{\sT}{\mathcal{T}} 
 
\newcommand{\sW}{\mathcal{W}}
\newcommand{\RR}{\mathbb{R}}
\newcommand{\ZZ}{\mathbb{Z}}

\newcommand{\NN}{\mathbb{N}} 
\newcommand{\TT}{\mathbb{T}} 
\newcommand{\PP}{\mathbb{P}} 
\newcommand{\EE}{\mathbb{E}}
\newcommand{\torus}{\TT^d_L}
\newcommand{\subbox}{\mathbb{B}^d}
\newcommand{\treespec}{g_{p,n}}
\newcommand{\naturals}{\NN}
\newcommand{\posint}{\ZZ_+}
\renewcommand{\d}{\mathrm{d}}
\newcommand{\pc}{p_{\mathrm{c}}}

\newcommand{\indicator}{\mathbf{1}}
\newcommand{\<}{\langle}
\renewcommand{\>}{\rangle}

\newcommand{\bits}{B}
\newcommand{\cube}{S^\bits}
\newcommand{\queries}{\sQ_\bits}

\begin{document}
\title{Critical speeding-up in dynamical percolation}
\author{Eren Metin El\c{c}i, Timothy M. Garoni}
\address{School of Mathematics, Monash University, Clayton, Victoria 3800, Australia}
\ead{\mailto{tim.garoni@monash.edu}}

\date{\today}

\maketitle

\maketitle
\begin{abstract}
    We study the autocorrelation time of the size of the cluster at the origin in discrete-time dynamical percolation. We focus on binary
    trees and high-dimensional tori, and show in both cases that this autocorrelation time is linear in the volume in the
    subcritical regime, but strictly sublinear in the volume at criticality. This establishes rigorously that the cluster size at the origin
    in these models exhibits critical speeding-up. The proofs involve controlling relevant Fourier coefficients. In the case of binary
    trees, these Fourier coefficients are studied explicitly, while for high-dimensional tori we employ a randomised algorithm argument
    introduced by Schramm and Steif in the context of noise sensitivity.
\end{abstract}


\section{Introduction}
Just as models in statistical mechanics typically exhibit divergent spatial scales near a critical point, stochastic dynamics associated
with such models also tend to display divergent time scales~\cite{HohenbergHalperin1977}. For example, Markov chains employed in Monte Carlo
algorithms at critical points typically exhibit relaxation times and stationary autocorrelation times which scale superlinearly with
volume~\cite{Sokal1997}; this well-known phenomenon is referred to as \emph{critical slowing-down}. By contrast, a less familiar possibility
is that time scales associated with certain observables scale sublinearly with volume at criticality. This phenomenon of \emph{critical
  speeding-up} has been observed experimentally (see
e.g.~\cite{BoukariBriggsShaumeyerGammon1990,GramsValldorGarstHemberger2014,TavoraRoschMitra2014,PatzLiLuoYangBudkoCanfieldPerakisWang2017}),
and also in Monte Carlo simulations~\cite{DengGaroniSokal2007,DengGaroniSokal2007_worm,ElciWeigel2013}.
 
In particular, in~\cite{DengGaroniSokal2007} the heat-bath dynamics of the critical Fortuin-Kasteleyn random-cluster model was simulated, on
2-dimensional and 3-dimensional tori, at a variety of cluster fugacities. The autocorrelations of certain observables were then studied. The
sum of squared cluster sizes, which is a natural estimator for the susceptibility, was studied in detail. It was observed generically that
its stationary autocorrelation function decays on a time scale strictly sublinear in the volume. Moreover, for a range of cluster
fugacities, sublinear scaling was also observed for its autocorrelation time, which quantifies the time scale between ``effectively
independent" samples. In particular, these behaviours were observed for the case of unit cluster fugacity, for which the stationary
distribution reduces to that of independent bond percolation. We will refer to this special case of the random-cluster model heat-bath
dynamics as discrete-time dynamical percolation; see section~\ref{subsec:dynamical percolation and autocorrelations} for a precise
definition.

Independently, closely-related questions were being addressed in the context of noise sensitivity of Boolean functions defined on the
discrete hypercube; see~\cite{GarbanSteif2015} for an overview.  The notion of noise sensitivity was introduced
in~\cite{BenjaminiKalaiSchramm1999}, and in essence it amounts to the fast decay of autocorrelation functions in dynamical percolation; for
precise definitions see section~\ref{sec:spectral representations}.  In particular, it was shown in~\cite{BenjaminiKalaiSchramm1999} that
the indicator for a left-right crossing of a rectangle in critical bond percolation on $\ZZ^2$ is noise sensitive. A quantitative version of
this result, bounding the rate of decay of these autocorrelations, was then presented in~\cite{SchrammSteif2010}. A key tool employed in
these results was Fourier analysis on the hypercube.

Our main aim in the current work is to study critical speeding-up using the same framework introduced for noise sensitivity. Whilst the
existing noise sensitivity literature focuses largely on Boolean functions, our focus will instead be on non-Boolean functions.
Specifically, we consider the size of the cluster containing the origin/root in dynamical percolation on binary trees and high-dimensional
tori. We show in both cases that this autocorrelation time is linear in the volume in the subcritical regime, but strictly sublinear in the
volume at criticality. In the case of high-dimensional tori, a key tool in our proof is a theorem due to Schramm and
Steif~\cite{SchrammSteif2010} which bounds Fourier coefficients by considering randomised algorithms.

\subsection{Outline}
The remainder of the article is organised as follows. Section~\ref{subsec:notation} lists some notational conventions.  In
section~\ref{subsec:dynamical percolation and autocorrelations}, we explicitly define dynamical percolation, in both discrete and continuous
time, and also recall standard definitions of the autocorrelation function and autocorrelation time. Statements of our main results are then
presented in section~\ref{subsec:Results}. Section~\ref{subsec:spectral form of autocorrelations} provides explicit spectral representations
of the autocorrelation functions and autocorrelation times, in both discrete-time and continuous time. Section~\ref{subsec:revealment
  predictability theorem} then discusses a slight variant of the revealment theorem of~\cite{SchrammSteif2010}. Our main results are then
proved in sections~\ref{sec:binary trees} and~\ref{sec:high dimensions}, which cover the case of binary trees and high-dimensional tori,
respectively. Finally, section~\ref{sec:revealment predictability proof} presents a self-contained proof of the theorem discussed in
section~\ref{subsec:revealment predictability theorem}.

\subsection{Notation}
\label{subsec:notation}
The set of non-negative integers will be denoted by $\naturals$, and $\posint:=\naturals\setminus0$. For any $n\in\posint$ we write
$[n]:=\{1,2,\ldots,n\}$. For a finite set $A$ and $k\in\naturals$ we denote by $\binom{A}{k}$ the set of all subsets of $A$ of cardinality
$k$, and we let $[A]:=[|A|]$. For vertices $x$ and $y$ in a graph, we will write $x \leftrightarrow y$ to denote that $x$ and $y$ belong to
the same connected component.

In addition to the standard asymptotic symbols $O$ and $o$ we find it convenient to write $f=\Omega(g)$ to denote $g=O(f)$ and $f=\omega(g)$
to denote $g=o(f)$. If both $f=O(g)$ and $f=\Omega(g)$ we will write $f=\Theta(g)$. We will also use the Vinogradov symbols, so that $f\ll
g$ is equivalent to $f=O(g)$, and $f\gg g$ is equivalent to $g=O(f)$.

\section{Definitions and results}
\label{sec:definitions and results}
\subsection{Dynamical percolation and autocorrelations}
\label{subsec:dynamical percolation and autocorrelations}
Let $S:=\{-1,1\}$, and let $\bits$ be any finite set. For a given $p\in(0,1)$, let
\begin{equation}
\label{eq:Bernoulli distribution}
    \pi_p(s):=p\,\indicator(s=+1)+(1-p)\,\indicator(s=-1), \qquad s \in S\;,
\end{equation}
denote the Bernoulli distribution on $S$, and let $\pi_{p,B}$ be the corresponding product measure on $S^B$, so that
\begin{equation}
\pi_{p,\bits}(x):=\prod_{i\in\bits}\pi_p(x_i)\;,\qquad \text{ for all } x\in S^B\;.
\end{equation}
We refer to the elements of $B$ as bits, and with respect to $x\in S^B$ we say $i\in B$ is open if $x_i=+1$ and closed if $x_i=-1$.

Now consider the stochastic matrix on state space $\cube$ with entries
\begin{equation}
P_{p,\bits}(x,y):= \frac{1}{|\bits|} \sum_{i\in\bits} \pi_p(y_i)\,\prod_{j\in  B\setminus i}\indicator(y_j=x_j)\;.
\end{equation}
This is the special case of the discrete-time heat-bath chain for the Fortuin-Kasteleyn model (see e.g.~\cite{Grimmett2009}) on a graph with
edge set $\bits$, when the cluster fugacity is $q=1$. In the special case $p=1/2$, it reduces to lazy simple random walk on the hypercube
$S^\bits$. The latter also corresponds to the infinite temperature limit of the discrete-time single-spin heat-bath process for the Ising
model, on a graph with vertex set $\bits$. For a given choice of $\bits$ and $p$, let $Z^{p,\bits}=(Z_s^{p,\bits})_{s\in\naturals}$ be a
stationary Markov chain with transition matrix $P_{p,\bits}$. We refer to the process $Z^{p,\bits}$ as discrete-time dynamical percolation
on $\bits$ with parameter $p$. We emphasise that for each $s\in\NN$, the state $Z^{p,\bits}_s$ has distribution $\pi_{p,\bits}$.

Discrete-time dynamical percolation has an intimate relationship to the dynamical percolation model first studied
in~\cite{HaggstrommPeresSteif1997}. The latter corresponds to a countably infinite family of independent continuous-time Markov chains on
$S$, each with generator $\Pi_p-I$, where $\Pi_p$ is the matrix on $S$ with entries $\Pi_p(s,s')=\pi_p(s')$. It is easily seen that the
subset of chains corresponding to a particular set of bits $\bits$ is a continuous time Markov chain on $\cube$ with generator
$Q_{p,\bits}=|\bits|(P_{p,\bits}-I)$; in particular, $Q_{p,\bits}$ is the rate-$|\bits|$ continuization (see e.g.~\cite[section
  20.1]{LevinPeres2017}) of $P_{p,\bits}$.

We are interested in the autocorrelations of the real-valued processes $f(Z_s^{p,\bits})$ induced by functions $f:\cube\to\RR$. The
corresponding \emph{autocorrelation function} is
\begin{equation}
\rho_{p,f}(s):=\frac{\cov(f(Z_0^{p,\bits}),f(Z_s^{p,\bits}))}{\var(f(Z_0^{p,\bits}))}, \qquad s\in\naturals\;,
\label{eq:discrete autocorrelation function}
\end{equation}
and \emph{autocorrelation time}\footnote{This is often referred to as the \emph{integrated} autocorrelation time~\cite{Sokal1997}.} is
\begin{equation}
\tau_{p,f}:=\frac{1}{2}\sum_{s\in\ZZ}\rho_{p,f}(|s|)\;.
\label{eq:discrete autocorrelation time}
\end{equation}
Analogous quantities can also be defined for continuous-time dynamical percolation. If $Y^{p,\bits}=(Y_t^{p,\bits})_{\ge0}$ is a stationary
continuous-time Markov chain with generator $Q_{p,\bits}$, then we set
\begin{equation}
\contrho_{p,f}(t):=\frac{\cov(f(Y_0^{p,\bits}),f(Y_t^{p,\bits}))}{\var(f(Y_0^{p,\bits}))}, \qquad t\in[0,\infty)\;,
\end{equation}
and
\begin{equation}
\conttau_{p,f}:=\int_0^{\infty}\contrho_{p,f_\bits}(t)\,\d t\;.
\end{equation}

One concrete motivation for interest in $\tau_{p,f}$ is that (see e.g.~\cite[equation~(9.2.11)]{MadrasSlade1996})
\begin{equation}
\var\left(\frac{1}{T}\sum_{s=0}^{T-1}f(Z_s^{p,\bits})\right) \sim \frac{2\tau_{p,f}}{T} \var(f(Z_0^{p,\bits})), \qquad T\to\infty\;.
\end{equation}
Therefore, if one wishes to estimate $\EE(f(Z_0^{p,\bits}))$ by simulating $Z^{p,\bits}$, the squared error is proportional to
$\tau_{p,f}$. By contrast, the variance of the sample mean of $m$ independent realisations of $f(Z_0^{p,\bits})$ is simply
$\var(f(Z_0^{p,\bits}))/m$. Since each such independent realisation consists of $|\bits|$ bit updates, in order to meaningfully compare
estimates from independent samples with those constructed using the Markov chain $Z^{p,\bits}$, one should set $T=m|\bits|$.
 
Now fix $p$, and consider an increasing sequence of bit sets, $\bits_n$, and a sequence of functions $f_n:S^{\bits_n}\to\RR$. It follows
that Markov-chain sampling asymptotically outperforms independent sampling in estimation of $\EE(f_n(Z_0^{p,\bits_n}))$ iff $\tau_{p,f_n}=
o(|\bits_n|)$ as $n\to\infty$. When a sequence of functions displays such behaviour, for a given value of $p$, we say it exhibits
\emph{speeding up}. In section~\ref{subsec:Results} we discuss the autocorrelation time for the cluster size at the origin in discrete-time
dynamical percolation on trees and high-dimensional tori. We find that while speeding up does not occur in the subcritical regime, it does occur at
criticality. This phenomenon is precisely what we shall mean by \emph{critical speeding-up}.

The autocorrelations of discrete-time and continuous-time dynamical percolation are closely related. In fact, it follows immediately
from lemma~\ref{lem:spectral autocorrelation forms} below that
\begin{equation}
    \tau_{p,f}=|\bits|\,\conttau_{p,f}-1/2\;.
 \label{eq:discrete vs continuous autocorrelation time}
\end{equation}
From equation~\eqref{eq:discrete vs continuous autocorrelation time}, it follows that a sequence of functions $f_n$ displays speeding up, at
parameter $p$, iff the corresponding sequence $\conttau_{p,f_n}=o(1)$ as $n\to\infty$.  Moreover, the following result, whose proof is
deferred to section~\ref{sec:spectral representations}, implies that the continuous-time and discrete-time autocorrelation functions have
closely related asymptotics, on non-trivial timescales.
\begin{lemma}
\label{lem:discrete time vs continuous time rho asymptotics}
Fix $p\in(0,1)$. Let $\bits_n$ be an increasing sequence of finite sets, and let $f_n:S^{\bits_n}\to\RR$ be a sequence of non-constant
functions. Then as $n\to\infty$, for any positive integer sequence $s_n$ satisfying $s_n=\omega(1)$ we have
$$
\rho_{p,f_n}(s_n) = \contrho_{p,f_n}(s_n/|\bits_n|) +o(1)\;.
$$
\end{lemma}

There is a direct relationship between the decay of the continuous-time autocorrelation function, and the notion of noise sensitivity.  Fix
$p\in(0,1)$, and let $X^{p,\bits}$ be a random element of $\cube$ with distribution $\pi_{p,\bits}$.  For given $\epsilon>0$, let
$X^{p,\bits}_{\epsilon}$ denote the following perturbation of $X^{p,\bits}$: independently for each bit $i$, with probability $1-\epsilon$
the $i$th bit of $X^{p,\bits}_{\epsilon}$ agrees with the $i$th bit of $X^{p,\bits}$, and with probability $\epsilon$ it is set to a new
random value chosen according to $\pi_p$, independent of its previous value, and independent of the values of all other bits.  Now consider,
again, an increasing sequence of finite sets $\bits_n$. A sequence of Boolean functions $f_n:S^{\bits_n}\to S$ is called \emph{noise
  sensitive}~\cite{GarbanSteif2015} if the covariance of $f_n(X^{p,\bits_n})$ and $f_n(X^{p,\bits_n}_{\epsilon})$ goes to zero as
$n\to\infty$ for all $\epsilon>0$. However, it is straightforward to verify directly that if $\epsilon=1-\e^{-t}$ then
$(X^{p,\bits},X^{p,\bits}_\epsilon)$ and $(Y^{p,\bits}_0,Y^{p,\bits}_t)$ have the same joint distribution. Therefore, the covariance of
$f_n(X^{p,\bits_n})$ and $f_n(X^{p,\bits_n}_{\epsilon})$ is precisely the covariance of $f_n(Y_0^{p,B})$ and $f_n(Y_t^{p,B})$, and the
correlation of $f_n(X^{p,\bits_n})$ and $f_n(X^{p,\bits_n}_{\epsilon})$ is precisely $\contrho_{p,f_n}(t)$. Provided that
$\var(f_n(X^{p,B_n}))$ is bounded away from zero, it then follows that $f_n$ is noise sensitive iff $\contrho_{p,f_n}(t)\to 0$ as
$n\to\infty$ for all fixed $t>0$.  We remark that although noise sensitivity is typically only considered in the specific case $p=1/2$, the
above discussion makes sense for any fixed $p\in(0,1)$. 

In the case of non-Boolean functions, where the variance is no longer bounded from above by 1, it is arguably more natural to consider the
correlation from the outset, rather than the covariance. In addition to the question of speeding up, we will be interested in the question of
sublinear autocorrelation decay, in the sense that there exists $c_n=o(|B_n|)$ such that $\rho_{p,f_n}(s_n)=o(1)$ for all $s_n=\omega(c_n)$.
We emphasise that speeding up and autocorrelation decay can occur at strictly different time scales. A simple example illustrating this is
presented below in equation~\eqref{eq:toy example f}.

\subsection{Results}
\label{subsec:Results}
Our main results concern autocorrelations of the size of the cluster containing a fixed vertex, in dynamical percolation on certain trees
and high-dimensional tori. The set of bits, $B$, then corresponds to the edge set of the tree or torus under consideration, and for $x\in
S^B$ drawn from $\pi_{p,B}$, we consider the connected components of the spanning subgraph defined by the open edges.  The size of a such a
connected component, or cluster, is then simply the number of vertices which it contains.

Our result for trees is as follows. By a \emph{binary tree}, we simply mean a rooted tree in which each vertex has at most 2 children.  We
say a binary tree is \emph{perfect} if each non-leaf vertex has precisely two children, and all leaves have the same depth. Recall (see
e.g.~\cite[theorem 1.8]{LyonsPeres2016}) that the critical threshold for bond percolation on the infinite-depth perfect binary tree is
$\pc=1/2$.
\begin{proposition}
  \label{prop:trees}
  For $n\in\posint$, let $(V_n,B_n)$ denote the perfect binary tree of depth $n$. For $p\in(0,1)$, consider discrete-time dynamical
  percolation on $(V_n,B_n)$, and let $f_n$ denote the size of the cluster containing the root.
 \begin{enumerate}[label=(\roman*)]
 \item\label{prop_part:tree off-critical rho} Let $s_n$ be a positive integer sequence. If $p\in(0,1)$ and $p\neq 1/2$ then
   $$
     \lim_{n\to\infty}\rho_{p,f_n}(s_n) = 
     \begin{cases}
     1, & \text{ if } s_n=o(|B_n|), \\
     0, & \text{ if } s_n=\omega(|B_n|). \\
     \end{cases}
     $$
 \item\label{prop_part:tree critical rho} Let $s_n$ be a positive integer sequence. Then
   $$
     \lim_{n\to\infty}\rho_{1/2,f_n}(s_n) = 
     \begin{cases}
     1, & \text{ if } s_n=o(|B_n|/n), \\
     0, & \text{ if } s_n=\omega(|B_n|/n). \\
     \end{cases}
     $$
 \item\label{prop_part:tree tau} As $n\to \infty$
   \begin{equation*}
     \dfrac{\tau_{p,f_n}}{|B_n|} \sim
      \begin{cases}
        \dfrac{1-2p}{2p(1-p)}\log\left(\dfrac{1-2p^2}{1-2p}\right), & p\in(0,1/2),\vspace{3mm}\\
        6\dfrac{\log(n)}{n}, & p=1/2,\vspace{3mm}\\
        \dfrac{2p-1}{(1-p)}\log\left(\dfrac{p}{2p-1}\right), & p\in(1/2,1).
      \end{cases}
     \end{equation*}
 \end{enumerate}
\end{proposition}
Part~\ref{prop_part:tree tau} of proposition~\ref{prop:trees} clearly illustrates the phenomenon of critical speeding up; the
autocorrelation time for the cluster at the root is linear in the volume for all off-critical $p$, but is strictly sublinear at
criticality. 

Our other example concerns dynamical percolation on high-dimensional tori. For integer $d\ge2$, let $\pc(d)$ denote the critical
threshold of bond percolation on $\ZZ^d$. For integer $L>2$, we denote by $\torus$ the $d$-dimensional discrete torus of side length $L$,
whose vertex set we take to be $[-L/2,L/2)^d\cap\ZZ^d$.
\begin{proposition}
\label{prop:high-dimensional tori}
Let $d,L\in\posint$ with $L>2$. Consider discrete-time dynamical percolation on $\torus$, and let $f_{d,L}$ denote the size of the cluster
connected to the origin.  For any fixed but sufficiently large $d$, we have as $L\to\infty$ that
$$
\tau_{p,f_{d,L}} = 
\begin{cases}
  \Theta(L^d), & 0<p<\pc(d)\;, \\
  O(L^{d-1/13}), & p=\pc(d)\;,
\end{cases}
$$
and for a positive sequence $t_L$ that
$$
\rho_{p,f_{d,L}}(\ceil{t_L L^d}) =
\begin{cases}
  \Theta(1), & t_L=o(1),\, p<\pc(d)\;,\\
  o(1), & t_L=\omega(L^{-1/13}),\, p=\pc(d)\;.
\end{cases}
$$
\end{proposition}
We note that the implied constants in the asymptotic statements in proposition~\ref{prop:high-dimensional tori} may in general depend on $p$
and $d$.

\section{Spectral representations}
\label{sec:spectral representations}
\subsection{Autocorrelations}
\label{subsec:spectral form of autocorrelations}
Given $p\in(0,1)$ and a finite set $\bits$, let 
$l^2(\pi_{p,\bits})$ denote the vector space $\RR^{\cube}$ endowed with the inner product
\begin{equation}
\<f,g\>_{p,\bits} := \sum_{x\in\cube}f(x)g(x)\pi_{p,\bits}(x), \qquad
\text{ for all } f,g\in\RR^{\cube}\;.
\end{equation}
Let
\begin{equation}
  \nu_p:=\sqrt{(1-p)/p}\;,
  \label{eq:nu definition}
\end{equation}
and for each $A\subseteq \bits$, define $\Psi_{A}^{p,\bits}:\cube\to\RR$ via
\begin{equation}
    \Psi_{A}^{p,\bits}(x) = \prod_{b\in A}\,x_i\,\nu_p^{x_i}, \qquad\text{ for all } x\in\cube\;.
    \label{eq:explicit eigenvectors}
\end{equation}
It can be easily verified that $\{\Psi_{A}^{p,\bits}\}_{A\subseteq \bits}$ is an orthonormal basis for $l^2(\pi_{p,\bits})$, and that, for
each $A\subseteq\bits$, $\Psi_A^{p,\bits}$ is an eigenvector of $P_{p,\bits}$ with eigenvalue
\begin{equation}
\lambda_A^{\bits}:= 1-\frac{|A|}{|\bits|}\;.
\label{eq:explicit eigenvalues}
\end{equation}
\begin{remark}
  \label{rem:no slowing down}
  Combining the explicit expression for the eigenvalues from equation~\eqref{eq:explicit eigenvalues} with general spectral results on
  autocorrelations in~\cite[equations (9.2.26) and (9.2.27)]{MadrasSlade1996} immediately implies that for any $p\in(0,1)$ and any function
  $f:\cube\to\RR$ we have $\tau_{f,p}=O(|\bits|)$ for large $|\bits|$. In particular, dynamical percolation does not suffer from critical
  slowing down.
\end{remark}

Given a non-constant function $f:\cube\to\RR$, it is convenient to introduce the random variable $\sW_{p,f}$, taking values in $[\bits]$,
with probability mass function
\begin{equation}
\PP(\sW_{p,f}=k)=\sum_{A\in\binom{\bits}{k}}
\frac{\<f,\Psi_A^{p,\bits}\>_{p,\bits}^2}{\var(f(Z_0^{p,\bits}))}\;.
\label{eq:energy sample definition}
\end{equation}
Autocorrelation functions and autocorrelation times for both the discrete time process $(Z_s^{p,\bits})_{s\in\naturals}$ and continuous time
process $(Y_t^{p,\bits})_{t\in[0,\infty)}$ can be conveniently expressed in terms of $\sW_{p,f}$. The following lemma, whose proof we omit,
  is a simple consequence of the spectral decomposition of $P_{p,\bits}$, using the explicit expressions from equations~\eqref{eq:explicit
    eigenvectors} and~\eqref{eq:explicit eigenvalues}.
\begin{lemma}
\label{lem:spectral autocorrelation forms}
Let $p\in(0,1)$, let $\bits$ be a finite set, and let $f:\cube\to\RR$ be a non-constant function. Then:
\begin{enumerate}[label=(\roman*)]
\item\label{lem_part:discrete rho in terms of W} $\rho_{p,f}(s) = \EE\left(1-\dfrac{\sW_{p,f}}{|\bits|}\right)^s$ for all $s\in\naturals$
\item\label{lem_part:discrete tau in terms of W} $\tau_{p,f} = |\bits|\,\EE(\sW_{p,f}^{-1})-1/2$ 
\item\label{lem_part:continuous rho in terms of W} $\tilde{\rho}_{p,f}(t) = \EE(\e^{-t\sW_{p,f}})$ for all $t\in[0,\infty)$
\item\label{lem_part:continuous tau in terms of W} $\tilde{\tau}_{p,f} = \EE(\sW_{p,f}^{-1})$
\end{enumerate}
\end{lemma}

As a simple example, fix $\gamma\in(0,1)$ and consider $f:S^{[n]}\to\RR$ defined by
\begin{equation}
  \label{eq:toy example f}
  f:= n^{-\gamma/2}\,\Psi_{\{1\}}^{p,[n]} + \sqrt{1-n^{-\gamma}}\,\Psi_{[n]}^{p,[n]}\;.
\end{equation}
If $p=1/2$, then $\Psi_{\{1\}}^{p,[n]}$ is referred to as a dictator function, and $\Psi_{[n]}^{p,[n]}$ is the parity function; see
e.g.~\cite{O'Donnell2014}. For any $p\in(0,1)$, however, equation~\eqref{eq:energy sample definition} yields
\begin{equation}
  \label{eq:toy example W}
  \PP(\sW_{p,f}=k)=n^{-\gamma}\,\indicator(k=1)+(1-n^{-\gamma})\,\indicator(k=n)\;.
\end{equation}
Lemma~\ref{lem:spectral autocorrelation forms} then implies that $\lim_{n\to\infty}\rho_{p,f}(1)=0$ and $\tau_{p,f}\sim n^{1-\gamma}$ as
$n\to\infty$. In particular, we see in this example that while the timescale at which $\rho_{p,f}$ decays to zero is of constant order in
$n$, the autocorrelation time diverges as a power law. This provides a simple example, as mooted at the end of section~\ref{subsec:dynamical
  percolation and autocorrelations}, that the timescales characterising autocorrelation decay and speeding up are not related in general.

Propositions~\ref{prop:trees} and~\ref{prop:high-dimensional tori} will be proved by establishing appropriate results for the distribution
of $\sW_{p,f}$, and then applying lemma~\ref{lem:spectral autocorrelation forms}. Lemma~\ref{lem:spectral autocorrelation forms} also
provides a straightforward proof of lemma~\ref{lem:discrete time vs continuous time rho asymptotics}.
\begin{proof}[Proof of lemma~\ref{lem:discrete time vs continuous time rho asymptotics}]
Let $n\in\posint$ and $s\in\naturals$. From lemma~\ref{lem:spectral autocorrelation forms} we have
\begin{align} 
\rho_{p,f_n}(s) &= \sum_{k\in[\bits_n]}\PP(\sW_{p,f}=k)\left(1-\frac{k}{|\bits_n|}\right)^s \\
&= \sum_{k=1}^{|\bits_n|-1} \e^{-ks/|\bits_n|}\PP(\sW_{p,f_n}=k)
\e^{-s\left(k/|\bits_n|\right)^2 h\left(k/|\bits_n|\right)}
\end{align}
where $h:[0,1)\to[0,\infty)$ is defined via 
\begin{equation}
h(x):=\frac{-x-\log(1-x)}{x^2}=\sum_{l=0}^{\infty}\frac{x^l}{l+2}
\end{equation}
and the simple identity $\e^{x}(1-x)=\e^{-x^2 h(x)}$ has been employed. This immediately yields the upper bound
\begin{align}
    \rho_{p,f_n}(s_n) &\le \sum_{k=1}^{|\bits_n|}\e^{-k s_n/|\bits_n|}\PP(\sW_{p,f_n}=k)\\
    &= \tilde{\rho}_{p,f_n}(s_n/|\bits_n|)\;.
\end{align}

Now let $c_n$ be a positive sequence satisfying $c_n=\omega(1)$ and $c_n=o(\sqrt{s_n})$, and set
$a_n:=\lfloor |\bits_n| c_n/s_n\rfloor$. Then
\begin{align}
\rho_{p,f_n}(s_n) 
&\ge \sum_{k=1}^{a_n}\e^{-k s_n/|\bits_n|}\PP(\sW_{p,f_n}=k)\e^{-s_n \left(k/|\bits_n|\right)^2 h\left(k/|\bits_n|\right)}\\
&\ge \sum_{k=1}^{a_n}\e^{-k s_n/|\bits_n|}\PP(\sW_{p,f_n}=k)\e^{-(c_n/\sqrt{s_n})^2 h\left(c_n/s_n\right)}\\
&= \tilde{\rho}_{p,f_n}(s_n/|\bits_n|)+o(1)
\end{align}
where the second inequality utilised the monotonicity of $h$. The stated result now follows.
\end{proof}

\subsection{Randomised bit-query algorithms}
\label{subsec:revealment predictability theorem}
In this section we discuss a slight variant of~\cite[theorem 1.8]{SchrammSteif2010}, which provides an upper bound for $\PP(\sW_{p,f}=k)$ in
terms of certain properties of a randomised algorithm.

By a ``bit querying algorithm", we picture a procedure which takes as input a bit string, $x\in\cube$, and outputs a subset of bits,
$J\subseteq\bits$, in such a way that bits are added one-by-one to $J$, and the choice of the next bit to add can depend on the states in
$x$ of the bits already examined. To make this more precise, we offer the following concrete model.

We define a \emph{query tree} on $\bits$ to be a binary tree of depth at least 1, with its non-leaf vertices labelled from $\bits$ and its
edges labelled from $S$, such that the following two conditions hold:
\begin{enumerate}[label=(\roman{*})]
    \item on each root-to-leaf path, each $i\in\bits$ occurs at most once
    \item no vertex has two descendent edges (i.e. edges connecting to descendants) with the same label
\end{enumerate}
We denote by $\queries$ the set of all query trees on $\bits$. By construction, $\queries$ is finite. Now let $x\in\cube$ and $t\in
\queries$.  If $v\in V(t)$ has label $b_v\in B$, then we abuse notation by writing $x(v):=x(b_v)$. We define the ``query path for $x$ in
$t$" to be the unique path $v_0,\ldots v_k$ in $t$ satisfying:
\begin{enumerate}[label=(\roman{*})]
    \item $v_0$ is the root
    \item for each $0\le i \le k-1$, the edge $v_i v_{i+1}$ has label $x(v_i)$
    \item $v_k$ has no descendent edge with label $x(v_k)$
\end{enumerate}
In words, given $x$, at each step, one looks at the bit in $x$ corresponding to the current vertex in $t$; if there is an outgoing edge with
this label, traverse it; if not, terminate.  Finally, we define $J:\queries\times\cube\to\sP(\bits)$ so that for each $(t,x)\in
\queries\times\cube$, the image $J(t,x)$ is the set of vertex labels appearing on the query path for $x$ in $t$.

\begin{remark}
Note that the above definition of a query tree is very closely related to the notion of a decision tree (see
e.g. ~\cite{O'Donnell2014}). However, we emphasise that, unlike a decision tree, the leaves of a query tree play no special role; in
particular, they are not labelled by the possible output values of any real-valued function.
\end{remark}

We now wish to consider $J$ acting on random elements of $\queries$ and $\cube$.  We begin with some helpful notation. For any
$I\subseteq\bits$ and $x\in\cube$, we denote by $x_I$ the restriction of $x$ to $I$; i.e. $x_I:I\to S$ is defined by $x_I(i)=x(i)$ for all
$i\in I$. Let $\bar{I}:=\bits\setminus I$. For any $y\in S^I$ and $z\in S^{\bar{I}}$, we define $y\oplus z$ to be the unique element of
$S^{\bits}$ satisfying
\begin{equation}
  \label{eq:oplus definition}
    (y\oplus z) (i) =
    \begin{cases}
    y(i), & i \in I,\\
    z(i), & i \in \bar{I}.
    \end{cases}
\end{equation}
Finally, if $g:\cube\to\RR$ then for any $I\subseteq\bits$ and $y\in S^I$ we define the map $g_{I|y}:\cube\to\RR$ via
\begin{equation}
    g_{I|y}(x) = g(y\oplus x_{\bar{I}}), \qquad\text{ for all } x\in\cube\;.
\end{equation}
We are now ready to define the two key quantities that will be used to bound $\PP(\sW_{p,f}=k)$. 
\begin{definition}
Let $B$ be a finite set, let $\pi$ be a probability measure on $\cube$, and let $\mu$ be a probability measure on $\queries$.  Let
$X\sim\pi$ and $T\sim\mu$ be independent. The \emph{revealment} of $\mu$ (with respect to $\pi$) is
\begin{equation}
    \delta_{\mu,\pi}:= \max_{i\in\bits}\, \PP[J(T,X)\ni i]
\end{equation}
and for any non-constant $f:\cube\to\RR$ the \emph{predictability} of $f$ (with respect to $\mu$ and $\pi$) is 
\begin{equation}
\epsilon_{\mu,\pi,f}:=\frac{\EE[\var_{\pi}(f_{J|X_{J}})]}{\var(f(X))}
\end{equation}
where $f_{J|X_J}$ abbreviates $f_{J(T,X)|X_{J(T,X)}}$, and the notation $\var_{\pi}(f_{J|X_{J}})$ means that for each fixed realisation of
$(T,X)$, one computes the variance with respect to $\pi$ of the map $f_{J|X_J}:\cube\to\RR$.
\end{definition}
\begin{remark}
We emphasise that since $f_{J|X_J}$ is a random map, $\var_{\pi}(f_{J|X_J})$ is a random variable on the space
$(\queries\times\cube,\mu\times\pi)$. We also note the bounds $\delta_{\mu,\pi}>0$, which follows from the requirement that
query trees have depth at least one, and $\epsilon_{\mu,\pi,f}\le 1$, which is established in corollary~\ref{cor:predictability is bounded
  above by 1}.
\end{remark}

We are now ready to state the following slight variant of the revealment theorem of~\cite{SchrammSteif2010}.
\begin{theorem}
\label{thm:revealment-predictability theorem}
Let $p\in(0,1)$ and let $\bits$ be a finite set. Let $\pi:=\pi_{p,\bits}$, and let $\mu$ be any distribution on $\queries$. For any
non-constant function $f:\cube\to\RR$ we have
$$
\PP(\sW_{p,f}=k) \le 2\,\epsilon_{\mu,\pi,f} + 2\,\delta_{\mu,\pi}\,k.
$$
\end{theorem}

We note that the statement of~\cite[theorem 1.8]{SchrammSteif2010} refers specifically to the case of randomised algorithms which exactly
determine $f$, meaning that the predictability term is identically zero. A generalisation to randomised algorithms which instead merely
approximate $f$ is, however, presented in~\cite[equation (2.5)]{SchrammSteif2010}. Our use of query trees is intended simply to provide a
concrete model for these randomised approximation schemes, where the focus is shifted away from approximating $f$, to simply constructing a
rule determining the set of queried bits.  We also note that~\cite[theorem 1.8]{SchrammSteif2010} and~\cite[equation
  (2.5)]{SchrammSteif2010} only refer to the special case that $p=1/2$, however the generalisation to arbitrary $p\in(0,1)$ is
straightforward. For completeness, we present a full proof of theorem~\ref{thm:revealment-predictability theorem} in
section~\ref{sec:revealment predictability proof}.

Combining theorem~\ref{thm:revealment-predictability theorem} with lemma~\ref{lem:spectral autocorrelation forms} provides the following
proposition, which is the main tool applied in section~\ref{sec:high dimensions}.  In short, it reduces
the problem of bounding $\tau_{p,f}$ and $\rho_{p,f}(s)$ to the problem of constructing a distribution $\mu$ on $\queries$ for which the
revealment and predictability both vanish sufficiently fast as $|\bits|$ increases.
\begin{proposition}
  \label{prop:revealment-predictability corollary}
Let $p\in(0,1)$ and let $\bits$ be a finite set with $n:=|\bits|>1$. Let $\pi:=\pi_{p,\bits}$, and let $\mu$ be any distribution on
$\queries$. For any non-constant function $f:\cube\to\RR$ we have
\begin{equation}
\tau_{p,f}
\le 
5\,n\,\log(n)\,\epsilon_{\mu,\pi,f} 
+ 3\,n\sqrt{\delta_{\mu,\pi}}\;,
\label{eq:revealment-predictability tau bound}
\end{equation}
and
\begin{equation}
\rho_{p,f}(\ceil{n\,t})\le 2\left(\frac{\epsilon_{\mu,\pi,f}}{t}+\frac{\delta_{\mu,\pi}}{t^2}\right)\;,
\qquad \text{ for all } t\in(0,\infty)\;.
\label{eq:revealment-predictability rho bound}
\end{equation}
\begin{proof}
For any $a_n\ge1$, we have from theorem~\ref{thm:revealment-predictability theorem} that
\begin{align}
\EE(\sW_{p,f}^{-1}) &\le 2\epsilon_{\mu,\pi,f}\sum_{k=1}^{\lfloor a_n\wedge n\rfloor}\frac{1}{k}
+2\delta_{\mu,\pi}(\lfloor a_n \rfloor \wedge n) +\frac{1}{a_n}\\
&\le 5\,\log(n)\,\epsilon_{\mu,\pi,f} + 2\delta_{\mu,\pi}\,a_n + a_n^{-1}
\end{align}
where in obtaining the second inequality we used the fact (see e.g.~\cite[equation~(6.60)]{GrahamKnuthPatashnik1994}) that for all $n\ge 2$
\begin{equation}
\sum_{k=1}^n\frac{1}{k}\le \frac{5}{2}\log(n)\;.
\end{equation}
The first stated result then follows by choosing $a_n=1/\sqrt{\delta_{\mu,\pi}}$ and applying lemma~\ref{lem:spectral autocorrelation
  forms}.

Moreover, theorem~\ref{thm:revealment-predictability theorem} implies that for all $t\ge0$
\begin{align}
    \EE(\e^{-t\,\sW_{p,f}})
    &\le 2\epsilon_{\mu,\pi,f} \sum_{k=1}^n \e^{-t\,k}+2\delta_{\mu,\pi}\sum_{k=1}^n k\,\e^{-t\,k}\\
    &\le 2\epsilon_{\mu,\pi,f}\, \frac{1}{\e^t-1} + 2\delta_{\mu,\pi}\,\frac{\e^t}{(\e^t-1)^2} \\
    &= \frac{2\epsilon_{\mu,\pi,f}}{t} + \frac{2\delta_{\mu,\pi}}{t^2}\;.
\end{align}
But from lemma~\ref{lem:spectral autocorrelation forms}, and the fact that $1-x\le \e^{-x}$ for all $x\in\RR$, we have
\begin{align}
    \rho_{p,f}(\ceil{n t}) &\le \EE\left(1-\frac{\sW_{p,f}}{n}\right)^{n t}\\
    &\le \EE(\e^{-t\sW_{p,f}})
\end{align}
which establishes the second stated result.
\end{proof}
\end{proposition}

\section{Binary Trees}
\label{sec:binary trees}
We now prove proposition~\ref{prop:trees}.  For $n\in\posint$, let $(V_n,B_n)$ denote the perfect binary tree of depth $n$, and let $o$
denote its root. For $p\in(0,1)$, we consider discrete-time dynamical percolation on $(V_n,B_n)$, and let $f_n$ denote the size of the
cluster containing $o$.  To simplify notation, we denote variance with respect to $\pi_{p,B_n}$ by $\var_p(\cdot)$, and set
$\sW_{p,n}:=\sW_{p,f_n}$. Lemma~\ref{lem:tree energy distribution} presents a simple exact expression for the distribution of $\sW_{p,n}$,
valid for all $n$.
\begin{lemma}
  \label{lem:tree energy distribution} Let $n\in\posint$ and $p\in(0,1)$. For all $k \in [\bits_n]$ we have 
  $$
    \var_p(f_n)\,\PP(\sW_{p,n} = k) = \nu_p^{2k}\,\sum_{d=1}^n\binom{d-1}{k-1}\,2^d\,\treespec(d)
    $$
    where $\treespec:[n]\to\RR$ is given by 
    $$
    \treespec(d):= \begin{cases}
      p^{2d}\left[\dfrac{1-(2p)^{n+1-d}}{1-2p}\right]^2, & \text{ if } p\neq 1/2,\\
      2^{-2d}\,(n+1-d)^2, & \text{ if } p=1/2.
    \end{cases}
    $$
    \begin{proof}
      Fix $n\in\posint$ and $p\in(0,1)$, and let $A\subseteq\bits_n$ be non-empty. To simplify notation, we abbreviate $\<.,.\>_{p,B_n}$ by
      $\<.,.\>$ and set $\Psi_A:=\Psi_A^{p,B_n}$. We begin by finding an expression for the coefficient $\<f_n,\Psi_A\>$.
      For $v\in V_n$, let $\sP_v$ denote the edge set of the unique path in $(V_n,B_n)$ between $o$ and $v$, and let $\delta_v$ be the
      depth of $v$. Then $v\leftrightarrow o$ iff every $e\in \sP_v$ is open. Therefore, for all $x\in S^{\bits_n}$ we have
      \begin{equation}
      f_n(x) = \sum_{v\in V_n}\prod_{e\in \sP_v}\left(\frac{1+x_e}{2}\right)\;.
      \end{equation}
      It follows that, if $X$ denotes a random element of $\bits_n$ with distribution $\pi_{p,\bits_n}$, then 
      \begin{equation}
        \begin{split}
        \<f_n,\Psi_A\> &= \EE\left(\sum_{v\in V_n} \prod_{e\in \sP_v}\left(\frac{1+X_e}{2}\right) \prod_{g\in A} X_g\,\nu_p^{X_g}\right) \\
        &= \sum_{v\in V_n}
        \prod_{e\in \sP_v\setminus A} \EE\left(\frac{1+X_e}{2}\right)
        \prod_{g\in A\setminus \sP_v} \EE(X_{g}\,\nu_p^{X_g}) \\
        &\qquad\times\prod_{h\in A\cap \sP_v} \left(\frac{\EE(X_h\,\nu_p^{X_h})}{2} + \frac{\EE(X_h^2\,\nu_p^{X_h})}{2}\right) \\
        &= \sum_{v\in V_n}\,p^{|\sP_v\setminus A|}\,(\sqrt{p(1-p)})^{|A\cap\sP_v|}\,\indicator(A\setminus\sP_v=\emptyset)\\
        &= \nu_p^{|A|}\,\sum_{v\in V_n}\indicator(\sP_v\supseteq A)\,p^{\delta_v}\;,
        \end{split}
        \label{eq:tree coefficient, intermediate}
      \end{equation}
      where we used the fact that for any $e\in B_n$ we have
      \begin{equation}
        \begin{split}
        \EE_p(X_e\, \nu_p^{X_e}) &= 0\;, \\
        \EE_p(X_e^2\,\nu_p^{X_e})&= 2\sqrt{p(1-p)}\;.
        \end{split}
        \end{equation}
      
      Now let
      \begin{equation}
      \sS_n:=\{A'\subseteq\bits_n \,:\, \exists v\in V_n \text{ with } \sP_v\supseteq A'\}\;,
      \end{equation}
      and
      \begin{equation}
      d_A:=\max\{\delta_v \,:\, v\in V_n \text{ and } \exists u\in V_n \text{ with } A\ni uv\}\;. 
      \end{equation}
      If $A\not\in\sS_n$, then equation~\eqref{eq:tree coefficient, intermediate}
      implies $\<f_n,\Psi_A\>=0$. Assume instead that $A\in\sS_n$, and let $v_A\in V_n$ be the (unique) deepest endpoint of the edges in
      $A$. Then $\sP_w\supseteq A$ for $w\in V_n$ iff $\sP_w\supseteq \sP_{v_A}$. For each $d_A+1\le k \le n$, there are $2^{k-d_A}$
      descendants of $v_A$ at depth $k$. Therefore, equation~\eqref{eq:tree coefficient, intermediate} implies
      \begin{equation}
        \begin{split}
          \<f_n,\Psi_A\> &= \indicator(A\in\sS_n)\,\nu_p^{|A|}\,p^{d_A}\,\sum_{k=0}^{n-d_A} (2p)^{k}\\
          &=\indicator(A\in\sS_n)\,
          \begin{cases}
            \nu_p^{|A|}\,p^{d_A}\dfrac{1-(2p)^{n-d_A+1}}{1-2p}, & p\neq 1/2\;,\\
            2^{-d_A}(n-d_A+1), & p=1/2\;.
          \end{cases}
          \end{split}
        \label{eq:tree coefficient, final}
      \end{equation}
      It then follows from equation~\eqref{eq:tree coefficient, final} that
      \begin{equation}
        \<f_n,\Psi_A\>^2 = \nu_p^{2|A|}\,\treespec(d_A)\,\indicator(A\in\sS_n)\;.
        \label{eq:squared tree coefficient}
      \end{equation}

      Finally, from equation~\eqref{eq:squared tree coefficient} we conclude that for any $k\in[\bits_n]$
      \begin{align}
        \sum_{A\in \binom{B_n}{k}}\,\<f_n,\Psi_A\>^2 &= \sum_{d=1}^n\,\sum_{\substack{v\in V_n\\\delta_v=d}}
        \,\sum_{\substack{A\in\sS_n\\|A|=k,\, v_A=v}}\,\nu_p^{2k}\,\treespec(d)\\
        &= \nu_p^{2k}\,\sum_{d=1}^n\,2^d\,\binom{d-1}{k-1}\,\treespec(d)\;.
      \end{align}
      The stated result now follows from equation~\eqref{eq:energy sample definition}.
    \end{proof}
  \end{lemma}

Armed with lemma~\ref{lem:tree energy distribution}, we can now prove proposition~\ref{prop:trees}.
\begin{proof}[Proof of proposition \ref{prop:trees}.]
  Asymptotic statements here refer to the limit $n\to\infty$ with $p$ fixed.  We begin by proving part~\ref{prop_part:tree off-critical
    rho}. To that end, let $z\in[0,1]$.  From lemma~\ref{lem:tree energy distribution} we
  have for any $p\in(0,1)$ that
  \begin{equation}
    \begin{split}
      \var_p(f_n)\,\EE(z^{\sW_{p,n}})&= \sum_{k\in[B_n]}\,z^k\,\nu_p^{2k}\,\sum_{d=1}^n\,2^d\,\binom{d-1}{k-1}\,\treespec(d)\\
      &= \frac{z\,\nu_p^2}{1+z\,\nu_p^2}\,\sum_{d=1}^n\,2^d\,(1+z\,\nu_p^2)^d\,\treespec(d)\;.
      \label{eq:tree MGF general}
      \end{split}
  \end{equation}
  For $p\neq 1/2$ it then follows that
  \begin{equation}
  \var_p(f_n)\,\EE(z^\sW)= \frac{z\,\nu_p^2}{(1+z\,\nu_p^2)(1-2p)^2}\,\sum_{d=1}^n\,[2p^2(1+z\,\nu_p^2)]^d\,[1-(2p)^{n+1-d}]^2\;.
  \end{equation}
  Now suppose $z_n\in[0,1]$ is a convergent sequence, and set $z_\infty:=\lim_{n\to\infty}z_n$. If $p\in(0,1/2)$,
  so that $(2p)^n$ is exponentially decaying for large $n$, we have
  \begin{equation}
    \lim_{n\to\infty}\,\var_p(f_n)\,\EE(z_{n}^{\sW_{p,n}})
    = \frac{z_\infty\,\nu_p^2}{(1+z_\infty\,\nu_p^2)(1-2p)^2}\frac{2p^2(1+z_\infty\,\nu_p^2)}{[1-2p^2(1+z_\infty\,\nu_p^2)]}\;,
    \label{eq:tree subcritical MGF}
  \end{equation}
  while if $p\in(1/2,1)$, so that $(2p)^n$ is exponentially increasing in $n$, we have
  \begin{equation}
    \lim_{n\to\infty}\,(2p)^{-2n}\,\var_p(f_n)\,\EE(z_n^{\sW_{p,n}})
    = \frac{z_\infty\,\nu_p^2}{(1+z_\infty\,\nu_p^2)(2p-1)^2}\frac{4p^2(1+z_\infty\,\nu_p^2)}{[1-z_\infty\,\nu_p^2]}\;.
    \label{eq:tree supercritical MGF}
  \end{equation}
  In particular, taking $z_n=1$ and employing the identity $1+\nu_p^2=1/p$ yields
  \begin{equation}
    \var_p(f_n)\sim
    \begin{cases}
      \dfrac{2p(1-p)}{(1-2p)^3}, & p\in(0,1/2),\vspace{2mm}\\

      \dfrac{4p^2(1-p)}{(2p-1)^3}\,(2p)^{2n}, & p\in(1/2,1).
    \end{cases}
    \label{eq:tree off-critical variance}
  \end{equation}
  Taking instead $z_n=\e^{-t_n}$ implies $z_\infty=1$ when $t_n=o(1)$, and $z_\infty=0$ when $t_n=\omega(1)$, and so equations~\eqref{eq:tree
    subcritical MGF}, \eqref{eq:tree supercritical MGF} and~\eqref{eq:tree off-critical variance} imply that for $p\neq1/2$ we have
  \begin{equation}
  \lim_{n\to\infty}\,\EE(\e^{-t_n\sW_{p,n}}) =
  \begin{cases}
    1, & t_n=o(1),\\
    0, & t_n=\omega(1).
  \end{cases}
  \label{eq:tree off-critical rhocont}
  \end{equation}

  Now let $s_n$ be a positive integer sequence, and set $t_n=s_n/|B_n|$. If $s_n=\omega(|B_n|)$ then combining equation~\eqref{eq:tree
    off-critical rhocont} with lemmas~\ref{lem:discrete time vs continuous time rho asymptotics} and~\ref{lem:spectral autocorrelation forms}
  implies
  \begin{equation}
    \lim_{n\to\infty}\,\rho_{p,f_n}(s_n) = 0\;.
  \end{equation}
  Suppose instead that $s_n=o(|B_n|)$. If $s_n=\omega(1)$, 
  then combining equation~\eqref{eq:tree
    off-critical rhocont} with lemmas~\ref{lem:discrete time vs continuous time rho asymptotics} and~\ref{lem:spectral autocorrelation forms}
  implies
  \begin{equation}
    \lim_{n\to\infty}\,\rho_{p,f_n}(s_n) = 1\;.
    \label{eq:tree rho small divergent times}
  \end{equation}
  If $s_n$ is instead bounded, then we can bound it above by $s_n'$ satisfying $s_n'=\omega(1)$ and $s_n'=o(|B_n|)$. By then applying
  equation~\eqref{eq:tree rho small divergent times} with $s_n$ replaced by $s_n'$, together with the fact that $\rho_{p,f_n}(s)$ is
  decreasing in $s$, and bounded above by 1, we see that in fact equation~\eqref{eq:tree rho small divergent times} also holds for bounded
  $s_n$.  We have therefore established part~\ref{prop_part:tree off-critical rho}.

  We now consider part~\ref{prop_part:tree critical rho}. From equation~\eqref{eq:tree MGF general} we have for $z\in[0,1]$ that
  \begin{equation}
    \var_{1/2}(f_n)\,\EE(z^{\sW_{1/2,n}}) = \frac{z}{2}\,\sum_{j=1}^n\,j^2\,\left(\frac{1+z}{2}\right)^{n-j}\;.
    \label{eq:tree critical MGF setup}
  \end{equation}
  Taking $z=1$ in equation~\eqref{eq:tree critical MGF setup} immediately yields
  \begin{equation}
  \var_{1/2}(f_n)=\frac{n(n+1)(2n+1)}{12}\;,
    \label{eq:tree critical variance}
  \end{equation}
  and it then follows that for $z\in[0,1)$
    \begin{multline}
      \EE(z^{\sW_{1/2,n}}) = \frac{12}{(1+n^{-1})(2+n^{-1})}\left(\frac{z}{n(1-z)} -2\frac{z(1+z)}{n^2(1-z)^2}\right.\\
      \left.+ \frac{z(1+z)(3+z)}{n^3(1-z)^3} \left[1-\left(\frac{1+z}{2}\right)^n\right]\right)\;.
    \label{eq:tree critical MGF}
    \end{multline}
    In particular, if $z=\e^{-t_n}$ with $t_n>0$ and $t_n\to c\in(0,+\infty]$, then
  \begin{equation}
    \lim_{n\to\infty} \EE(\e^{-t_n\sW_{1/2,n}}) = 0\;.
  \end{equation}

  Suppose instead that $t:=t_n=o(1)$. It then follows from equation~\eqref{eq:tree critical MGF} that 
     \begin{align}
       \EE(\e^{-t \sW_{1/2,n}}) &= \frac{12}{(1+n^{-1})(2+n^{-1})}
       \left[\frac{t}{(\e^t-1)}\frac{1}{n t} -2\frac{t^2(\e^t+1)}{(\e^t-1)^2}\frac{1}{(n t)^2}\right.\nonumber
       \\
       &\left.\quad + \frac{t^3 (\e^t+1)(3\e^t+1)}{(\e^t-1)^3}\frac{1}{(n t)^3}\left[1-\left(\frac{1+\e^{-t}}{2}\right)^n\right]\right]\\
       &= \frac{12}{(1+n^{-1})(2+n^{-1})}\left(
       \frac{1}{n t}[1+O(t)]-2\frac{1}{(n t)^2}[2-t+O(t^2)]\right.\nonumber\\
       &\left.\quad+\frac{1}{(n t)^3}[8-2t+O(t^2)]\left[1-\left(\frac{1+\e^{-t}}{2}\right)^n\right]
       \right)\;,
       \label{eq:tree critical MGF small t}
     \end{align}
     where in performing the above Taylor approximation it is helpful to recall that $t/(\e^t-1)$ is the exponential generating function of
     the Bernoulli numbers (see e.g.~\cite{GrahamKnuthPatashnik1994}). If $t=\omega(n^{-1})$ it now follows immediately
     from equation~\eqref{eq:tree critical MGF small t} that 
     \begin{equation}
     \lim_{n\to\infty} \EE(\e^{-t_n\,\sW_{1/2,n}}) = 0\;.
     \end{equation}
     Conversely, now suppose that $t=o(n^{-1})$. Then
      \begin{equation}
        1-\left(\frac{1+\e^{-t}}{2}\right)^n = \frac{nt}{2}-\frac{(nt)^2}{8}-\frac{nt^2}{8}+\frac{(nt)^3}{48} + O[n^2t^3 \vee (nt)^4]\;, 
        \label{eq:tree critical MGF exponential expansion}
      \end{equation}
      and so the last term in equation~\eqref{eq:tree critical MGF small t} satisfies, as $n\to\infty$,
     \begin{equation}
       \frac{1}{(n t)^3}[8-2t+O(t^2)]\left[1-\left(\frac{1+\e^{-t}}{2}\right)^n\right]
          =
          \frac{4}{(nt)^2}-\frac{1}{nt}-\frac{2}{n^2 t}+ \frac{1}{6} + o(1)\;.
          \label{eq:tree critical MGF small t last term}
     \end{equation}
      Combining equation~\eqref{eq:tree critical MGF small t last term} with equation~\eqref{eq:tree critical MGF small t} then yields
      \begin{equation}
        \lim_{n\to\infty} \EE(\e^{-t\sW_{1/2,n}})=1\;,
      \end{equation}
      We have therefore established that
      \begin{equation}
        \lim_{n\to\infty}\,\EE(\e^{-t_n\sW}) =
        \begin{cases}
          1, & t_n=o(n^{-1})\\
          0, & t_n=\omega(n^{-1}).
        \end{cases}
        \label{eq:tree critical rhocont}
      \end{equation}
      Just as part~\ref{prop_part:tree off-critical rho} was shown above to follow from~\eqref{eq:tree off-critical rhocont}, a
      precisely analogous argument shows that part~\ref{prop_part:tree critical rho} follows from~\eqref{eq:tree critical rhocont}.

      Finally, we turn attention to part~\ref{prop_part:tree tau}. From lemma~\ref{lem:tree energy distribution} we have for any $p\in(0,1)$
      that
  \begin{align}
    \var_p(f_n) \,\EE(\sW_{p,n}^{-1}) &= \sum_{d=1}^n\,2^d\,g_p(d)\,\sum_{k=1}^d\,\binom{d-1}{k-1}\frac{\nu_p^{2k}}{k}\\
    &=  \sum_{d=1}^n\,2^d\,g_p(d)\,\left(\frac{p^{-d}-1}{d}\right)\;.
    \label{eq:tree autocorrelation time general}
  \end{align}
  Suppose $p\neq1/2$. It follows from equation~\eqref{eq:tree autocorrelation time general} that
  \begin{multline}
    (1-2p)^2\,\var_p(f_n)\,\EE(\sW_{p,n}^{-1}) \\
    = \sum_{d=1}^n\,\frac{1}{d}\left[(2p)^d+(2p)^{2n+2-d}-2(2p)^{n+1}-(2p^2)^d-(2p)^{2n+2}2^{-d}+2(2p)^{n+1}p^d\right]\;.
    \label{eq:tree noncritical autocorrelation time}
  \end{multline}
  Now observe that for any $x\in(0,1)$ we have
  \begin{equation}
    \sum_{d=1}^\infty\frac{x^d}{d}=-\log(1-x)\;.
    \label{eq:tree log sum}
  \end{equation}
  Equations~\eqref{eq:tree noncritical autocorrelation time} and~\eqref{eq:tree log sum} then imply that for $p\in(0,1/2)$ we have
  \begin{equation}
    \lim_{n\to\infty}(1-2p)^2\,\var_p(f_n)\,\EE(\sW_{p,n}^{-1}) = \log(1-2p^2)-\log(1-2p)\;,
    \label{eq:tree sub-critical var tau}
  \end{equation}
  while for $p\in(1/2,1)$ we have
  \begin{equation}
    \lim_{n\to\infty}(2p)^{-2n-2}(1-2p)^2\,\var_p(f_n)\,\EE(\sW_{p,n}^{-1}) = \log\left(\frac{2p}{2p-1}\right) -\log(2)\;.
    \label{eq:tree super-critical var tau}    
  \end{equation}

  In the critical case, equation~\eqref{eq:tree autocorrelation time general} yields
  \begin{align}
    \var_{1/2}(f_n)\,\EE(\sW_{1/2,n}^{-1}) &= \sum_{d=1}^n\frac{(n+1-d)^2}{d}(1-2^{-d}) \\
    &= n^2\log(n) + O(n^2)\;.
    \label{eq:tree critical var tau}
  \end{align}
  Part~\ref{prop_part:tree tau} then follows by combining equations~\eqref{eq:tree off-critical variance}, \eqref{eq:tree sub-critical var
    tau}, and \eqref{eq:tree super-critical var tau}, and equations \eqref{eq:tree critical variance} and~\eqref{eq:tree critical var tau},
  with part~\ref{lem_part:discrete tau in terms of W} of lemma~\ref{lem:spectral autocorrelation forms}.
\end{proof}

\section{High-dimensional tori}
\label{sec:high dimensions}
In this section we will prove proposition~\ref{prop:high-dimensional tori}. To that end, let $d,L\in\posint$ with $L\ge3$.  For $p\in(0,1)$,
we consider discrete-time dynamical percolation on $\torus$. Let $B_L:=E(\torus)$. We denote expectation and variance with respect to
$\pi_{p,B_L}$ by $\EE_{p}$ and $\var_{p}$. Let $C_L$ denote the cluster containing the origin, so that $f_{d,L}=|C_L|$.  To simplify
notation we set $\sW_{p,L}:=\sW_{p,f_{d,L}}$, $\tau_{p,L}:=\tau_{p,f_{d,L}}$, $\contrho_{p,L}:=\contrho_{p,f_{d,L}}$ and
$\rho_{p,L}:=\rho_{p,f_{d,L}}$, and for all $A\subseteq B_L$ we set $\Psi_A:=\Psi_A^{p,B_L}$.  To avoid conflict with the subscripts
required when applying theorem~\ref{thm:revealment-predictability theorem}, we shall also henceforth abbreviate $f:=f_{d,L}$.

It will be helpful on occasion to relate bond percolation on $\torus$ to bond percolation on $\ZZ^d$. We shall denote the probability
measure and expectation with respect to the latter by $\PP_{p,\ZZ^d}$ and $\EE_{p,\ZZ^d}$. We view $\ZZ^d$ as a graph in which vertices of
Euclidean distance 1 are adjacent, and we make no notational distinction between this graph and its vertex set. For bond percolation on
$\ZZ^d$ we shall denote the cluster containing the origin by $C_{\ZZ^d}$.
\begin{proof}[Proof of proposition~\ref{prop:high-dimensional tori}]
  Asymptotic statements here refer to the limit $L\to\infty$, for fixed values of $p$ and $d$, and implied constants may in general
  depend on $p$ and $d$.  We first consider the subcritical case, and therefore fix $p\in(0,\pc)$. Let $e\in B_L$ denote the edge from the
  origin to the unit vector $(1,0,\ldots,0)$ on the positive $x$ axis. Then from equation~\eqref{eq:energy sample definition} we have
  \begin{equation}
    \EE(\sW_{p,L}^{-1}) = \sum_{k\in[B_L]}\,\frac{1}{k}\,\sum_{A\in\binom{B_L}{k}}\,\frac{\<f,\Psi_A\>^2}{\var_p(f)}
    \ge \frac{\<f,\Psi_{e}\>^2}{\var_p(f)}\;,
    \label{eq:tori subcritical tau setup}
  \end{equation}
  and for any $t\ge0$
  \begin{equation}
    \EE(\e^{-t\,\sW_{p,L}}) = \sum_{k\in[B_L]}\,\e^{-t\,k}\,\sum_{A\in\binom{B_L}{k}}\,\frac{\<f,\Psi_A\>^2}{\var_p(f)}
    \ge \frac{\<f,\Psi_{e}\>^2}{\var_p(f)}\,\e^{-t}\;.
    \label{eq:tori subcritical rho setup}
  \end{equation}

  We begin by lower bounding $\<f,\Psi_{e}\>^2$. Employing the notation of equation~\eqref{eq:oplus definition}, with $I=\{e\}$, we have
  \begin{align}
    \<f,\Psi_e\> &= \sum_{\omega\in S^{B_L\setminus e}}\,\pi_{p,B_L\setminus e}(\omega)\,\sum_{a\in S^e}
    \pi_{p}(a)\,a\,\nu_p^a\,f(a\oplus\omega)\\
    &= \sum_{\omega\in S^{B_L\setminus e}}\,\pi_{p,B_L\setminus e}(\omega)\,\sum_{a\in S^e}
    \sqrt{p(1-p)}\, a\, \,f(a\oplus\omega)\\
    &= \sqrt{p(1-p)}\sum_{\omega\in S^{B_L\setminus e}}\,\pi_{p,B_L\setminus e}(\omega)\,\left[f(\omega^e)-f(\omega_e)\right]
  \end{align}
  where for $\omega\in S^{B_L\setminus e}$ we define $\omega^e,\omega_e\in S^{B_L}$ so that $\omega^e(e)=+1$ and $\omega_e(e)=-1$, while
  $\omega^e(e')=\omega_e(e')=\omega(e')$ for all $e'\in B_L\setminus e$.

  Now define
  \begin{equation}
    \sN:=\{\omega\in S^{B_L\setminus e}\;:\;\omega(e')=-1 \text{ for all } e'\in E\setminus e \text{ such that $0$ is an endpoint of
      $e'$}\}\;.
  \end{equation}
  Observe that for all $\omega\in S^{B_L\setminus e}$ we have $f(\omega^e)\ge f(\omega_e)$, while if $\omega\in\sN$ we have
  $f(\omega^e)\ge2$ and $f(\omega_e)=1$. Also observe that $\pi_{p,B_L}(\sN)=(1-p)^{2d-1}$. It follows that
  \begin{equation}
    \<f,\Psi_e\>\ge \sqrt{p(1-p)}\,(1-p)^{2d-1}\;.
    \label{eq:tori subcritical fourier bound}
  \end{equation}
  
  We can also find a constant order upper bound for $\var_p(f)$, as follows. We begin with
  \begin{equation}
    \var_p(f)\le \EE_{p}(|C_L|^2)\;.
    \label{eq:tori critical variance bounded by second moment}
  \end{equation}
  But it follows from~\cite[proposition 2.1]{HeydenreichVanDerHofstad2007} that
  \begin{equation}
    \EE_{p}(|C_L|^2) \le \EE_{p,\ZZ^d}(|C_{\ZZ^d}|^2)\;,
    \label{eq:coupling torus to infinite lattice}
  \end{equation}
  and the behaviour of $|C_{\ZZ^d}|$ when $p<\pc$ is well understood. In particular, it is known (see e.g.~\cite[theorem 6.75]{Grimmett1999}) that
  for all $p<\pc$, there exists $\lambda_p>0$ such that for all $n\ge1$
  \begin{equation}
    \PP_{p,\ZZ^d}(|C_{\ZZ^d}|\ge n)\le \e^{-\lambda_p n}\;.
  \end{equation}
  Consequently,
  \begin{equation}
    \EE_{p,\ZZ^d}(|C_{\ZZ^d}|^2)\le \sum_{n=1}^\infty n^2\,\e^{-\lambda_p n} \le
    \frac{\e^{\lambda_p}(\e^{\lambda_p}+1)}{(\e^{\lambda_p}-1)^3}<\infty\;.
    \label{eq:tori subcritical variance bound}
  \end{equation}
  It then follows from equations~\eqref{eq:tori subcritical fourier bound}, \eqref{eq:tori critical variance bounded by second moment},
  \eqref{eq:coupling torus to infinite lattice} and~\eqref{eq:tori subcritical variance bound} that for all $0<p<\pc$ we have
  \begin{equation}
    \frac{\<f,\Psi_{e}\>^2}{\var_p(f)} \ge p(1-p)^{4d-1}\,\frac{(\e^{\lambda_p}-1)^3}{\e^{\lambda_p}(\e^{\lambda_p}+1)} > 0\;.
    \label{eq:subcritical torus fourier/variance bound}
  \end{equation}

  Combining equations~\eqref{eq:tori subcritical tau setup} and~\eqref{eq:subcritical torus fourier/variance bound} with
  lemma~\ref{lem:spectral autocorrelation forms} then implies that $\tau_{p,L}=\Omega(|B_L|)$, so we conclude from remark~\ref{rem:no
    slowing down} that in fact $\tau_{p,L}=\Theta(|B_L|)$. Similarly, if $s_L$ is an integer sequence with $s_L=o(|B_L|)$, then combining
  equations~\eqref{eq:tori subcritical rho setup} and~\eqref{eq:subcritical torus fourier/variance bound} with lemma~\ref{lem:spectral
    autocorrelation forms} implies $\contrho_{p,L}(s_L/|B_L|)=\Omega(1)$. For $s_L$ satisfying $s_L=o(|B_L|)$ and $s_L=\omega(1)$ we can
  then conclude from lemma~\ref{lem:discrete time vs continuous time rho asymptotics} that $\rho_{p,L}(s_L)=\Omega(1)$, and monotonicity
  then implies that in fact $\rho_{p,L}(s_L)=\Omega(1)$ for any $s_L=o(|B_L|)$. But $\rho_{p,L}(s)\le 1$ for any $s\in\NN$, and so in fact
  $\rho_{p,L}(s_L)=\Theta(1)$.

  We now turn attention to the critical case of $p=\pc$. Our strategy is to describe a distribution on $\sQ_{B_L}$, and then apply
  theorem~\ref{thm:revealment-predictability theorem}. For concreteness, we impose the lexicographic ordering on the vertices of
  $\ZZ^d$. For $r\in\posint$, we denote the cube in $\ZZ^d$ of radius $r$ centred at $v\in\ZZ^d$ by
  \begin{equation}
    \subbox_r(v):=\{y\in\ZZ^d:\|y-v\|_{\infty}\le r\}\;,
  \end{equation}
  where $\|\cdot\|_\infty$ denotes the sup norm on $\RR^d$. We abbreviate $\subbox_r:=\subbox_r(0)$. We fix $\kappa\in(0,1)$, and define the
  sequences $\aL:=\floor{L^\kappa}$ and $\bL:=\floor{L^{\kappa/3}}$. For $r\in[\aL]$, we consider the subgraph of $\ZZ^d$ induced by
  $\subbox_r$, and let $E_r$ denote its edge set. In what follows, we will assume $L$ is sufficiently large that $a<L/2$,
  which then guarantees $\subbox_r$ is a proper subset of $V(\torus)$.

  We construct a query tree as follows. Let $r\in[\aL]$ and $x\in S^{B_L}$, and consider the subgraph $\sG_r$ of $\subbox_r$ corresponding
  to $x_{E_r}\in S^{E_r}$. Beginning at the smallest vertex on the boundary $\partial\subbox_r$, perform breadth-first search on
  $\sG_r$. Naive breadth-first search would in principle require as input the full adjacency list of $\sG_r$, which would require querying
  in advance each edge in $E_r$. We instead query just the necessary edges on the fly as follows. From vertex $v$, query each edge $uv\in
  E_r$ which has not been previously queried, then mark $uv$ as queried. If $x_{uv}=1$, add $u$ to the adjacency list of $v$ and vice
  versa. Once the breadth-first search from a given point in $\partial\subbox_r$ has terminated, start a new breadth-first search from the
  next smallest vertex in $\partial\subbox_r$ which has not previously been visited by a breadth-first search. When this process terminates,
  the set of edges which have been queried consists precisely of those edges in $E_r$ which have at least one endpoint belonging to a
  component in $\sG_r$ that is connected to $\partial\subbox_r$. Let $J_1$ denote the set of these edges.

  At the conclusion of the above process it will be known whether or not $0\leftrightarrow\partial\subbox_r$; in particular,
  $0\leftrightarrow\partial\subbox_r$ holds iff $J_1$ contains an edge incident to the origin which is open.  If
  $0\leftrightarrow\partial\subbox_r$, then the above breadth-first search procedure is recommenced from the vertices on $\partial\subbox_r$
  which are connected to $0$, but now with respect to all edges in $\torus$. Therefore, if $0\leftrightarrow\partial\subbox_r$, then the
  component of $x$ containing $0$ will be determined exactly. The set of edges queried in this second step is precisely all those edges in
  $B_L\setminus E_r$ which have at least one endpoint in the component of $x$ containing $0$. We denote this set by $J_2$.  If
  $0\not\leftrightarrow\partial\subbox_r$, then $J_2=\emptyset$.  The total set of edges queried in the above two-stage process is therefore
  $J=J_1\cup J_2$. Note that each edge in $B_L$ is queried at most once. For any fixed $r\in[\aL]$, the above procedure defines a
  deterministic query tree on $B_L$, which we denote by $t_r$.
  
  Now let $\sR$ be a uniformly random element of $[\aL]$, and let $\sT:=t_\sR$. Denote the distribution of $\sT$ by $\mu$, and let $X$ be
  independent of $\sT$ with distribution $\pi:=\pi_{\pc,B_L}$. We will bound the predictability and revealment corresponding to $\mu$ and
  $\pi$. We begin by considering the predictability. By construction, on the event $0\leftrightarrow\partial\subbox_\sR$, the set of queried
  bits, $J$, is sufficiently large that $f_{J|X_J}$ is a constant function of the bits in $B_L\setminus J$, so that
  $\var_\pi(f_{J|X_J})=0$. While on $0\not\leftrightarrow\partial\subbox_\sR$, the set of queried bits, $J=J_1$, contains a cut-set of
  closed edges in $E_\sR$ which separates the union of the components connected to $\partial\subbox_\sR$, from a set $V_0\subset\subbox_\sR$
  containing the origin. Therefore, on $0\not\leftrightarrow\partial\subbox_\sR$, regardless of the values taken by $X$ on $B_L\setminus J$
  we have $f_{J|X_J} <|\subbox_\sR|=\sR^d$, and so $\var_{\pi}(f_{J|X_J})\le\sR^{2d}$. It follows that
  \begin{align}
    \EE\left(\var_\pi(f_{J|X_J})\right) &=
    \EE\left(\var_\pi(f_{J|X_J})\indicator(0\not\leftrightarrow\partial\subbox_\sR)\right)\\
    &\le \sum_{r\in[\aL]}\,\EE
    \left(r^{2d}\,\indicator(0\not\leftrightarrow\partial\subbox_\sR)\indicator(\sR=r)\right)\\
    &\le
    \sum_{r\in[\aL]}\,r^{2d}\,\PP(\sR=r)\\
    &\le \frac{1}{\aL}\int_{0}^{\aL+1} t^{2d}\, \d t \\
    &\ll L^{2\,d\, \kappa}\;.
    \label{eq:critical torus predictability numerator bound}
  \end{align}

  To lower bound $\var(f(X))$, we begin by noting that Chebyshev's inequality implies that
  \begin{equation}
    \var(f(X))\ge \chi^2\,\PP[f(X)\ge 2\,\chi]\;,
    \label{eq:critical torus Chebyshev bound}
  \end{equation}
  where $\chi:=\EE(f(X))$. But~\cite[corollary 2.2]{HeydenreichVanDerHofstad2011}, which builds upon results
  in~\cite{BorgsChayesVanDerHofstadSladeSpencerI2005,BorgsChayesVanDerHofstadSladeSpencerII2005}, implies that for sufficiently large $d$ we
  have
  \begin{equation}
    \chi\gg L^{d/3}\;,
    \label{eq:critical torus chi bound}
  \end{equation}
  and
  \begin{equation}
    \PP[f(X)\ge 2\,\chi] \gg L^{-d/6}\;.
    \label{eq:critical torus tail bound}
  \end{equation}
  Combining equations~\eqref{eq:critical torus chi bound} and~\eqref{eq:critical torus tail bound} with equation~\eqref{eq:critical torus
    Chebyshev bound} then yields
  \begin{equation}
    \var(f(X))\gg L^{d/2}\;.
    \label{eq:critical torus variance bound}
  \end{equation}
  We therefore conclude from equations~\eqref{eq:critical torus predictability numerator bound} and~\eqref{eq:critical torus variance bound}
  that
  \begin{equation}
    \epsilon_{\mu,\pi,f} \ll L^{2d\kappa-d/2}\;.
    \label{eq:critical torus predictability bound}    
  \end{equation}
  Taking $\kappa<1/4$ implies that $\epsilon_{\mu,\pi,f}=o(1)$.
  
  We now bound the revealment, and begin by considering $J_1$. For $u\in\subbox_{\aL}$, define
  \begin{align}
    \sA_u &:= \{r\in[\aL]: \|u\|_\infty \le r < \|u\|_\infty + \bL\}\;,
  \end{align}
  and write $\bar{\sA_u}:=[a]\setminus\sA_u$. By construction, $\PP(J_1\ni uv)=0$ for all $uv\in E(\torus)\setminus E_{\aL}$, and so we
  instead suppose that $uv\in E_{\aL}$. Then
  \begin{align}
    \PP(J_1\ni uv) &= \PP(J_1\ni uv,\sR\in \sA_u\cup\sA_v) + \PP(J_1\ni uv,\sR\not\in \sA_u\cup\sA_v)\\
    &\le \PP(\sR\in\sA_u) + \PP(\sR\in\sA_v) + \PP(J_1\ni uv,\sR\in\bar{\sA_u},\sR\in\bar{\sA_v})\\
    &\ll L^{-2\kappa/3} + \PP(J_1\ni uv,\sR\in\bar{\sA_u},\sR\in\bar{\sA_v})\;.
    \label{eq:small R J1 bound}
  \end{align}
  But since, by construction, $\sR<\|u\|_\infty$ implies $J_1\not\ni uv$, we have
  \begin{equation}
    \PP(J_1\ni uv,\sR\in\bar{\sA_u},\sR\in\bar{\sA_v}) = \PP(J_1\ni uv, \sR\ge\|u\|_\infty + \bL, \sR\ge\|v\|_\infty + \bL)\;.
  \end{equation}
  And since $J_1\ni uv$ occurs only if at least one of the events $u\leftrightarrow \partial\subbox_\sR$ and $v\leftrightarrow
  \partial\subbox_\sR$ occur, it then follows by the union bound that
  \begin{align}
    \PP(J_1\ni uv,\sR\in\bar{\sA_u},\sR\in\bar{\sA_v})
    &\le2\max_{w\in\subbox_{\aL}}\,\PP(w\leftrightarrow\partial\subbox_\sR,\sR\ge\|w\|_\infty+\bL)\\
    &=2\max_{w\in\mathbb{B}_{\aL}^d}\,\sum_{r=\|w\|_\infty+\bL}^{\aL}\,\PP(w\leftrightarrow\partial\subbox_r)\,\PP(\sR=r)\;.
    \label{eq:critical torus J1 bound for large R}
  \end{align}
  
  Let $w\in\subbox_a$. If $r\ge \|w\|_\infty+\bL$, then $w\in\subbox_{r-\bL}(0)$, and so $\subbox_{\bL}(w)\subseteq\subbox_r(0)$, which
  implies
  \begin{align}
    \PP(w\leftrightarrow \partial\subbox_r(0)) &\le \PP(w\leftrightarrow\partial\subbox_{\bL}(w))\\
    &= \PP(0\leftrightarrow \partial\subbox_{\bL}(0))\\
    &\ll \frac{1}{\bL^2}\;,
    \label{eq:critical torus arm event bound}
  \end{align}
  where the second line follows by translational invariance, and the final step follows from~\cite[theorem 1]{KozmaNachmias2011JAMS}
  provided $d$ is sufficiently large. Applying the bound from equation~\eqref{eq:critical torus arm event bound} to
  equation~\eqref{eq:critical torus J1 bound for large R}, we then conclude from equation~\eqref{eq:small R J1 bound} that for all $uv\in
  B_L$ we have
  \begin{equation}
    \PP(J_1\ni uv)\ll L^{-2\kappa/3}\;.
    \label{eq:critical torus J1 bound}
  \end{equation}

  We now consider $J_2$. An edge can belong to $J_2$ only if it is outside $\subbox_\sR$ and if at least one of its endpoints is connected
  to the origin. It follows that for any $uv\in B_L$ we have
  \begin{align}
    \PP(J_2\ni uv) &\le \PP(\sR\le b) + \PP(J_2\ni uv, \sR>b)\\
    &\ll L^{-2\kappa/3} + \PP(\sR<\|u\|_\infty\vee\|v\|_\infty,\{u\leftrightarrow0\}\cup\{v\leftrightarrow0\}, \sR>b)\\
    &\ll L^{-2\kappa/3} + \PP(b<\sR\le\|u\|_\infty,u\leftrightarrow0)
    +\PP(b<\sR\le \|v\|_\infty,v\leftrightarrow0)\\
    &\ll L^{-2\kappa/3} + \max_{w\in \torus}\,\PP(b<\sR\le \|w\|_\infty, w\leftrightarrow 0)\\
    &\ll L^{-2\kappa/3} + \max_{\substack{w\in \torus\\ \|w\|_\infty>b}}\,\PP(w\leftrightarrow 0)\\
    &\ll L^{-2\kappa/3} + b^{-(d-2)} + L^{-2d/3}
  \end{align}
  where the final step makes use of~\cite[corollary 1.4]{HutchcroftMichtaSlade2023}, and is valid for all sufficiently large $d$. We
  conclude that for all sufficiently large $d$,
  \begin{equation}
    \PP(J_2\ni uv) \ll L^{-2\kappa/3}\;.
    \label{eq:critical torus J2 bound}
  \end{equation}
  Therefore, since $J=J_1\cup J_2$, combining equations~\eqref{eq:critical torus J1 bound} and~\eqref{eq:critical torus J2 bound} we obtain
  \begin{equation}
    \delta_{\mu,\pi} \ll L^{-2\kappa/3}\;.
    \label{eq:critical torus revealment bound}
  \end{equation}

  Finally, combining equations~\eqref{eq:critical torus predictability bound} and~\eqref{eq:critical torus revealment bound} with
  equation~\eqref{eq:revealment-predictability tau bound} of proposition~\ref{prop:revealment-predictability corollary} then implies that
  for any $\xi>0$
  \begin{equation}
    \frac{\tau_{\pc,f}}{L^d} \ll L^{2d\kappa+\xi-d/2} + L^{-\kappa/3}\;.
  \end{equation}
  Choosing $\kappa=3(d-2\xi)/2(6d+1)$ and $\xi=(d-2)/26$ then yields
  \begin{equation}
        \frac{\tau_{\pc,f}}{L^d} \ll L^{-1/13}\;,
  \end{equation}
  as required. Similarly, combining equations~\eqref{eq:critical torus predictability bound} and~\eqref{eq:critical torus revealment bound}
  with equation~\eqref{eq:revealment-predictability rho bound} of proposition~\ref{prop:revealment-predictability corollary}, and choosing
  $\kappa=3d/2(1+6d)$, shows that for all $t>0$ we have
  \begin{equation}
    \rho_{\pc,L}(\ceil{t L^d})\ll \frac{L^{-1/13}}{t} + \left(\frac{L^{-1/13}}{t}\right)^2\;.
  \end{equation}
  We then conclude for any $t=t_L=\omega(L^{-1/13})$ that
  \begin{equation}
    \lim_{L\to\infty} \rho_{\pc,L}(\ceil{t L^d})=0\;.
  \end{equation}
\end{proof}
  
\section{Proof of the revealment-predictability theorem}
\label{sec:revealment predictability proof}
For completeness, we now provide a self-contained proof of theorem~\ref{thm:revealment-predictability theorem}. In fact, with no extra
effort, we consider the slightly more general situation in which $\pi$ is an arbitrary positive product measure on $\cube$. The proof
follows closely the arguments presented in~\cite[section 2]{SchrammSteif2010}.

\begin{proof}[Proof of theorem~\ref{thm:revealment-predictability theorem}]
  Let $\bits$ be a finite set. For each $i\in\bits$, let $\pi_i:S\to(0,1)$ be a positive probability measure on $S$,
and let $\pi$ be the corresponding product measure on $S^B$, defined by
\begin{equation}
\pi(x):=\prod_{i\in\bits}\pi_i(x_i)\;,\qquad \text{ for all } x\in S^B\;.
\end{equation}
Let $\nu_i=\sqrt{\pi_i(-1)/\pi_i(+1)}$, and define $\psi_i:S\to\RR$ so that $\psi_i(s)=s\nu_i^s$ for $s\in S$. 
Then, for each $A\subseteq\bits$ define ${\Psi_A:\cube\to\RR}$ via 
\begin{equation}
    \Psi_A(x) = \prod_{i\in A}\psi_i(x_i), \qquad\text{ for all } x\in\cube\;.
\end{equation}
It can be easily verified that $\{\Psi_A\}_{A\subseteq\bits}$ is an orthonormal basis for $l^2(\pi)$. In what follows, all inner products
and lengths are taken with respect to the inner product on $l^2(\pi)$. 

Let $\mu$ be any probability measure on $\queries$. Let $X$ and $T$ be independent, with $X$ having distribution $\pi$, and $T$ having
distribution $\mu$.  We abbreviate $J(T,X)$ by $J$. Recalling the definitions in
section~\ref{subsec:revealment predictability theorem}, note that if $g,h:\cube\to\RR$, then $\<g_{J|X_J},h_{J|X_J}\>$ is a random variable
on the space $(\queries\times \cube,\mu\times\pi)$. We have the following useful lemma, whose proof we defer until the end of this section.
\begin{lemma}\label{lem:expectation of random inner product}
For any $g,h:\cube\to\RR$ we have $\EE\,[\<g_{J|X_J},h_{J|X_J}\>]=\<g,h\>$.
\end{lemma}

Now let $f:\cube\to\RR$. To simplify notation, we set $\delta:=\delta_{\mu,\pi}$ and $\epsilon:=\epsilon_{\mu,\pi,f}$. Our central task is
to establish the following inequality 
\begin{equation}
\label{eq:Schramm-Steif inequality}
\sum_{A\in\binom{\bits}{k}} \<f,\Psi_A\>^2
\le 2\, \EE \left[\var_{\pi}(f_{J|X_J})\right]
+ 
2\,k\, \var(f(X))\, \delta\;.
\end{equation}
First note that both sides of the above inequality are invariant upon replacing $f$ by $f-\<f,\Phi_\emptyset\>$. It therefore suffices to
establish equation~\eqref{eq:Schramm-Steif inequality} under the assumption that $\<f,\Phi_\emptyset\>=0$. Under this assumption we have
$\|f\|^2=\var(f(X))$.

Fix $k\in[\bits]$, and define $g:\cube\to\RR$ via
\begin{equation}
g:= \sum_{A\in\binom{\bits}{k}}\<f,\Psi_A\>\,\Psi_A\;,
    \label{eq:g definition}
\end{equation}
and note that
\begin{equation}
\label{eq:length of g in terms of f}
\<f,g\>=\sum_{A\in\binom{\bits}{k}} \<f,\Psi_A\>^2 = \|g\|^2\;.
\end{equation}
Equation \eqref{eq:Schramm-Steif inequality} holds trivially if $g=0$, so assume $g\neq0$.

From lemma~\ref{lem:expectation of random inner product} we have
\begin{align}
    \|g\|^2 &= \EE\<g_{J|X_J},f_{J|X_J}\>\\
    &= \EE\<g_{J|X_J},\Psi_{\emptyset}\>\<f_{J|X_J},\Psi_{\emptyset}\>
    +\EE
    \<g_{J|X_J},f_{J|X_J}-\<f_{J|X_J},\Psi_{\emptyset}\>\Psi_{\emptyset}\>
   \\
    &\le 
    \EE\<g_{J|X_J},\Psi_{\emptyset}\>\<f_{J|X_J},\Psi_{\emptyset}\>
    +\EE\,
    \|g_{J|X_J}\|\,\|f_{J|X_J}-\<f_{J|X_J},
    \Psi_{\emptyset}\>\Psi_{\emptyset}\| \\
    &\le
    \sqrt{\EE[\<g_{J|X_J},\Psi_{\emptyset}\>^2]\,
          \EE[\<f_{J|X_J},\Psi_{\emptyset}\>^2]}
          \\
  &\qquad+\sqrt{\EE\,[\|g_{J|X_J}\|^2]
            \,\EE[\|f_{J|X_J}-\<f_{J|X_J},\Psi_{\emptyset}\>\Psi_{\emptyset}\|^2]}
\label{eq:Cauchy-Schwarz on g}          
\end{align}
where the penultimate step follows from the Cauchy-Schwarz inequality applied to $\<\cdot,\cdot\>$, and the final step follows from the
Cauchy-Schwarz inequality applied to $\EE(\cdot)$.

For any $h:\cube\to\RR$, one has $\|h-\<h,\Psi_{\emptyset}\>\Psi_{\emptyset}\|^2 = \var_{\pi}(h)$,
and lemma~\ref{lem:expectation of random inner product} implies
\begin{equation}
\label{eq:expected length of random restriction}
\EE\|h_{J|X_J}\|^2=\|h\|^2\;.
\end{equation}
Applying these observations to equation~\eqref{eq:Cauchy-Schwarz on g} yields
\begin{equation}
\label{eq:g norm bound after variance inserted}
\|g\|^2\le \sqrt{\EE[\<g_{J|X_J},\Psi_{\emptyset}\>^2]\,
          \EE[\<f_{J|X_J},\Psi_{\emptyset}\>^2]} 
          +\|g\|\,\sqrt{\EE[\var_{\pi}(f_{J|X_J})]}\;.
\end{equation}

Now observe that for any $h:\cube\to\RR$ 
\begin{equation}
    \<h,\Psi_{\emptyset}\>^2 \le \|h\|^2 - \sum_{A\in\binom{\bits}{k}} \<h,\Psi_A\>^2 \le \|h\|^2\;.
    \label{eq:bound on zeroth mode}
\end{equation}
Applying equations~\eqref{eq:expected length of random restriction} and~\eqref{eq:bound on zeroth mode} with $h=f_{J|X_J}$, and recalling
$\|f\|^2=\var(f(X))$, then results in
\begin{equation}
\label{eq:g norm bound after extracting f norm}
\|g\|^2 \le
\sqrt{\var(f(X))\,\EE[\<g_{J|X_J},\Psi_{\emptyset}\>^2]}
          +\|g\|\,\sqrt{\EE[\var_{\pi}(f_{J|X_J})]}\;,
\end{equation}
and dividing by $\|g\|$ and squaring\footnote{Making use of $(a+b)^2\le 2(a^2+b^2)$} yields
\begin{equation}
\|g\|^2 \le 
2\var(f(X))\,\frac{\EE[\<g_{J|X_J},\Psi_{\emptyset}\>^2]}{\|g\|^2}
+2\,\EE[\var_{\pi}(f_{J|X_J})]\;.
\label{eq:penultimate bound on length of g}
\end{equation}

It remains only to bound $\EE[\<g_{J|X_J},\Psi_{\emptyset}\>^2]$. Applying equations~\eqref{eq:expected length of random restriction}
and~\eqref{eq:bound on zeroth mode} now to $h=g_{J|X_J}$ we obtain
\begin{align}                
\EE \<g_{J|X_J},\Psi_{\emptyset}\>^2
&\le\|g\|^2 - \sum_{A\in\binom{\bits}{k}} \EE \<g_{J|X_J},\Psi_A\>^2\\
&= \sum_{A\in\binom{\bits}{k}} \EE [\<g,\Psi_A\>^2 - \<g_{J|X_J},\Psi_A\>^2]\;.
\label{eq:bound for zeroth mode of randomised g}
\end{align}

One can easily verify that for any $A,I\subseteq \bits$ and $y\in\cube$
we have
\begin{equation}           
\label{eq:restriction of Psi}
(\Psi_A)_{I|y_I} = \Psi_{A\cap I}(y)\Psi_{A\setminus I}\;,
\end{equation}
and it follows, for any $h:\cube\to\RR$, that
\begin{equation}
\label{eq:spectral expansion for restrictions}
h_{I|y_I}= \sum_{C\subseteq\bits} \<h,\Psi_C\>\Psi_{C\cap I}(y)\Psi_{C\setminus I}\;.
\end{equation}
Consequently, for any $A\subseteq\bits$
\begin{equation}
\label{eq:spectral coefficient for restricted g}
\<g_{I|y_I},\Psi_A\>=\sum_{C\in\binom{\bits}{k}}\<g,\Psi_C\>\Psi_{C\cap I}(y)\<\Psi_{C\setminus I},\Psi_A\>\;.
\end{equation}
In particular, if $|A|=k$ we obtain
\begin{equation}
\label{eq:size k spectral coefficient for restricted g}
\<g_{I|y_I},\Psi_A\>= \<g,\Psi_A\>\indicator(A\cap I=\emptyset)\;,
\end{equation}
and therefore
\begin{align}
    \EE[\<g,\Psi_A\>^2 -\<g_{J|X_J},\Psi_A\>^2] &= \<g,\Psi_A\>^2\,\PP(A\cap J\neq\emptyset)\\
    &\le \<g,\Psi_A\>^2\,\sum_{i\in A}\PP(J\ni i)\\
    &\le \<g,\Psi_A\>^2\,k\,\delta\;.
\label{eq:revealment bound on random g coefficients}    
\end{align}
Combining equations~\eqref{eq:revealment bound on random g coefficients} and~\eqref{eq:bound for zeroth mode of randomised g} yields
\begin{equation}
\label{eq:revealment bound on zeroth mode of g}
\EE \<g_{J|X_J},\Psi_{\emptyset}\>^2 
\le k\,\delta\, \sum_{A\in\binom{\bits}{k}} \<g,\Psi_A\>^2 
= k\,\delta\,\|g\|^2\;,
\end{equation}
and equation~\eqref{eq:Schramm-Steif inequality} then follows upon combining equation~\eqref{eq:revealment bound on zeroth mode of g} with
equations~\eqref{eq:penultimate bound on length of g} and~\eqref{eq:length of g in terms of f}. If $f$ is non-constant, one can then divide
both sides of equation~\eqref{eq:Schramm-Steif inequality} by $\var(f(X))$, and specialising to the case $\pi=\pi_{p,\bits}$ for some
$p\in(0,1)$ then yields the statement of theorem~\ref{thm:revealment-predictability theorem}.
\end{proof}

\begin{remark}
  A comparison of equation~\eqref{eq:Schramm-Steif inequality} with~\cite[equation (2.5)]{SchrammSteif2010} shows that where
the latter contains a factor $\|f\|^2$, the former contains the slightly improved factor $\var(f(X))$. This minor difference is a
simple consequence of the argument applied immediately below~\eqref{eq:Schramm-Steif inequality}.
\end{remark}

It remains now only to prove lemma~\ref{lem:expectation of random inner product}.
\begin{proof}[Proof of lemma~\ref{lem:expectation of random inner product}]
Let $t\in \queries$ and $I\subseteq\bits$. From the definition of query paths, if for some $y\in S^I$ and $z\in S^{\bar{I}}$ we have
$J(t,y\oplus z)=I$, then in fact $J(t,y\oplus z')=I$ for all $z'\in S^{\bar{I}}$. It follows that if
\begin{equation}
\Sigma_t^I:=\{y\in S^I \,:\, \exists x\in \cube \text{ with } x_I=y, J(t,x)=I\}\;,
\end{equation}
then
\begin{align}
    \EE\,[\<g_{J|X_J},h_{J|X_J}\>]
    &=
    \sum_{t\in \queries} \sum_{I\subseteq\bits}\sum_{y\in\Sigma_t^I}
    \sum_{z\in S^{\bar{I}}}\PP(T=t,X_I=y,X_{\bar{I}}=z)\,\<g_{I|y},h_{I|y}\>\\
    &=
    \sum_{t\in \queries} \sum_{I\subseteq\bits}\sum_{y\in\Sigma_t^I}
    \PP(T=t,X_I=y)\,\<g_{I|y},h_{I|y}\>\\
    &=
    \sum_{t\in \queries} \sum_{I\subseteq\bits}\sum_{y\in\Sigma_t^I}
    \PP(T=t,X_I=y)\,\sum_{x'\in\cube}\pi(x')g_{I|y}(x')h_{I|y}(x')\\
    &=
    \sum_{t\in \queries} \sum_{I\subseteq\bits}\sum_{y\in\Sigma_t^I}
    \PP(T=t,X_I=y)\,\nonumber
    \\
    &\qquad\times
    \sum_{y'\in S^I}
    \sum_{z'\in S^{\bar{I}}}\PP(X_I=y',X_{\bar{I}}=z')
    g_{I|y}(y'\oplus z')h_{I|y}(y'\oplus z')\\
    &=
    \sum_{t\in \queries} \sum_{I\subseteq\bits}\sum_{y\in\Sigma_t^I}
    \PP(T=t,X_I=y)\,\nonumber
    \\
    &\qquad\times
    \sum_{y'\in S^I}
    \sum_{z'\in S^{\bar{I}}}\PP(X_I=y',X_{\bar{I}}=z')
    g(y\oplus z')h(y\oplus z')\\
    &=
    \sum_{t\in \queries} \sum_{I\subseteq\bits}\sum_{y\in\Sigma_t^I}
    \sum_{z\in S^{\bar{I}}}
    \PP(T=t,X_I=y)\,
    \PP(X_{\bar{I}}=z)
    g(y\oplus z)h(y\oplus z)\\
    &=
    \sum_{t\in \queries}\sum_{x\in \cube}\PP(T=t,X=x)\,g(x)h(x)\\
    &=
    \sum_{x\in \cube}\pi(x)\,g(x)h(x)\\
    &=\<g,h\>\;.
\end{align}
\end{proof}

\begin{corollary}
  \label{cor:predictability is bounded above by 1}
  Let $p\in(0,1)$ and let $\bits$ be a finite set. Let $\pi:=\pi_{p,\bits}$, and let $\mu$ be any distribution on $\queries$. For any
  non-constant function $f:\cube\to\RR$ we have $\epsilon_{\mu,\pi,f}\le1$.
  \begin{proof}
    Suppose $\EE(f)=0$, so that $\<f,f\>= \var(f(X))$. Since
    \begin{equation}
      \var_{\pi}(f_{J|X_J}) = \sum_{\substack{A\subseteq[n]\\A\neq\emptyset}}\,\<f_{J|X_J},\Psi_A\>^2 \le \<f_{J|X_J},f_{J|X_J}\>\;,
    \end{equation}
    it follows from lemma~\ref{lem:expectation of random inner product} that 
    \begin{align}
      \EE[\var_{\pi}(f_{J|X_J})] \le \var(f(X))\;.
      \label{eq:predictability bound}
    \end{align}
    Since both sides of equation~\eqref{eq:predictability bound} are unchanged upon adding a constant to $f$, the assumption of mean zero
    can be relaxed, and the stated result follows.
  \end{proof}
\end{corollary}

\ack{The authors thank Andrea Collevecchio and Greg Markowsky for many useful conversations during the early stages of this work, and also
  Bob Griffiths for useful comments on an earlier draft. This research was supported by the Australian Research Council Centre of Excellence
  for Mathematical and Statistical Frontiers (Project no. CE140100049), and the Australian Research Council's Discovery Projects funding
  scheme (Project No.s DP140100559, DP180100613 and DP230102209).}

\section*{References}
\bibliographystyle{unsrt}

\begin{thebibliography}{10}

\bibitem{HohenbergHalperin1977}
P~C Hohenberg and B~I Halperin.
\newblock {Theory of dynamic critical phenomena}.
\newblock {\em Reviews of Modern Physics}, 49:435--479, 1977.

\bibitem{Sokal1997}
Alan~D. Sokal.
\newblock Monte carlo methods in statistical mechanics: Foundations and new
  algorithms.
\newblock In P.~Cartier C.~de Witt-Morette and A.~Folacci, editors, {\em
  Functional Integration: Basics and Applications}, pages 131--192. Plenum
  Press, New York, 1997.

\bibitem{BoukariBriggsShaumeyerGammon1990}
Hac{\'e}ne Boukari, Matthew~E. Briggs, J.~N. Shaumeyer, and Gammon~Robert W.
\newblock {Critical Speeding Up Observed}.
\newblock {\em {Physical Review Letters}}, 65:2654--2657, 1990.

\bibitem{GramsValldorGarstHemberger2014}
Christoph~P Grams, Martin Valldor, Markus Garst, and Joachim Hemberger.
\newblock Critical speeding-up in the magnetoelectric response of spin-ice near
  its monopole liquid-gas transition.
\newblock {\em Nature communications}, 5:4853, 2014.

\bibitem{TavoraRoschMitra2014}
Marco Tavora, Achim Rosch, and Aditi Mitra.
\newblock Quench dynamics of one-dimensional interacting bosons in a disordered
  potential: Elastic dephasing and critical speeding-up of thermalization.
\newblock {\em {Physical Review Letters}}, 113:010601, 2014.

\bibitem{PatzLiLuoYangBudkoCanfieldPerakisWang2017}
A.~Patz, T.~Li, L.~Luo, X.~Yang, S.~Bud'ko, P.~C. Canfield, I.~E. Perakis, and
  J.~Wang.
\newblock {Critical speeding up of nonequilibrium electronic relaxation near
  nematic phase transition in unstrained $Ba(Fe_{1-x}Co_{x})_2As_{2}$}.
\newblock {\em Physical Review B}, 95:165122, 2017.

\bibitem{DengGaroniSokal2007}
Youjin Deng, Timothy~M Garoni, and Alan~D Sokal.
\newblock Critical speeding-up in the local dynamics of the random-cluster
  model.
\newblock {\em {Physical Review Letters}}, 98:230602, 2007.

\bibitem{DengGaroniSokal2007_worm}
Y~Deng, T~M Garoni, and A.~D. Sokal.
\newblock {Dynamic Critical Behavior of the Worm Algorithm for the Ising
  Model}.
\newblock {\em Physical Review Letters}, 99:110601, 2007.

\bibitem{ElciWeigel2013}
Eren~Metin El{\c{c}}i and Martin Weigel.
\newblock Efficient simulation of the random-cluster model.
\newblock {\em Physical Review E}, 88:033303, 2013.

\bibitem{GarbanSteif2015}
{Christophe Garban and Jeffrey E. Steif}.
\newblock {\em {Noise Sensitivity of Boolean Functions and Percolation}}.
\newblock {Cambridge University Press}, Cambridge, 2015.

\bibitem{BenjaminiKalaiSchramm1999}
Itai Benjamini, Gil Kalai, and Oded Schramm.
\newblock {Noise sensitivity of Boolean functions and applications to
  percolation}.
\newblock {\em Publications Math\'ematiques de l'IH\'ES}, 90:5--43, 1999.

\bibitem{SchrammSteif2010}
{Oded Schramm and Jeffrey E. Steif}.
\newblock Quantitative noise sensitivity and exceptional times for percolation.
\newblock {\em {Annals of Mathematics}}, 171:619--672, 2010.

\bibitem{Grimmett2009}
Geoffrey Grimmett.
\newblock {\em The Random-Cluster Model}.
\newblock Springer, New York, 2009.

\bibitem{HaggstrommPeresSteif1997}
{O. H{\"a}ggstr{\"o}mm and Y. Peres and J. E. Steif}.
\newblock Dynamical percolation.
\newblock {\em Ann. Inst. Henri Poincar{\'e}, Probab. et Stat}, 33:497--528,
  1997.

\bibitem{LevinPeres2017}
{David A. Levin and Yuval Peres}.
\newblock {\em Markov chains and mixing times}.
\newblock American Mathematical Society, Providence, 2017.

\bibitem{MadrasSlade1996}
Neal Madras and Gordon Slade.
\newblock {\em The Self-Avoiding Walk}.
\newblock Birkh\"auser, Berlin, 1996.

\bibitem{LyonsPeres2016}
Russell Lyons and Yuval Peres.
\newblock {\em {Probability on Trees and Networks}}.
\newblock Cambridge University Press, Cambridge, 2016.

\bibitem{O'Donnell2014}
Ryan O'Donnell.
\newblock {\em Analysis of Boolean functions}.
\newblock Cambridge University Press, Cambridge, 2014.

\bibitem{GrahamKnuthPatashnik1994}
{Ronald Graham, Donald Knuth, and Oren Patashnik}.
\newblock {\em {Concrete Mathematics}}.
\newblock {Addison–Wesley}, {Westford, Massachusetts}, 1994.

\bibitem{HeydenreichVanDerHofstad2007}
Markus Heydenreich and Remco Van Der~Hofstad.
\newblock Random graph asymptotics on high-dimensional tori.
\newblock {\em {Communications in Mathematical Physics}}, 270:335--358, 2007.

\bibitem{Grimmett1999}
G.~Grimmett.
\newblock {\em Percolation}.
\newblock Springer, New York, 2nd edition, 1999.

\bibitem{HeydenreichVanDerHofstad2011}
Markus Heydenreich and Remco van~der Hofstad.
\newblock Random graph asymptotics on high-dimensional tori ii: volume,
  diameter and mixing time.
\newblock {\em Probability theory and related fields}, 149:397--415, 2011.

\bibitem{BorgsChayesVanDerHofstadSladeSpencerI2005}
Christian Borgs, Jennifer~T Chayes, Remco van~der Hofstad, Gordon Slade, and
  Joel Spencer.
\newblock {Random subgraphs of finite graphs: I. The scaling window under the
  triangle condition}.
\newblock {\em Random Structures \& Algorithms}, 27:137--184, 2005.

\bibitem{BorgsChayesVanDerHofstadSladeSpencerII2005}
Christian Borgs, Jennifer~T Chayes, Remco van~der Hofstad, Gordon Slade, and
  Joel Spencer.
\newblock {Random subgraphs of finite graphs: II. The lace expansion and the
  triangle condition}.
\newblock {\em {Annals of Probability}}, 33:1886--1944, 2005.

\bibitem{KozmaNachmias2011JAMS}
Gady Kozma and Asaf Nachmias.
\newblock Arm exponents in high dimensional percolation.
\newblock {\em Journal of the American Mathematical Society}, 24:375--409,
  2011.

\bibitem{HutchcroftMichtaSlade2023}
Tom Hutchcroft, Emmanuel Michta, and Gordon Slade.
\newblock {High-dimensional near-critical percolation and the torus plateau}.
\newblock {\em {Annals of Probability}}, 51:580--625, 2023.

\end{thebibliography}

\end{document}